\documentclass[letterpaper,10pt]{amsart}
\usepackage{amsmath,amssymb,amsxtra,amsthm, amstext, amscd,amsfonts,fancyhdr, hyperref,enumerate}
\usepackage[OT2,T1]{fontenc}
\DeclareSymbolFont{cyrletters}{OT2}{wncyr}{m}{n}
\DeclareMathSymbol{\Sha}{\mathalpha}{cyrletters}{"58}

\newcommand{\bC}{{\mathbb{C}}}

\newcommand{\bF}{{\mathbb{F}}}

\newcommand{\bN}{{\mathbb{N}}}

\newcommand{\bQ}{{\mathbb{Q}}}
\newcommand{\bR}{{\mathbb{R}}}

\newcommand{\bZ}{{\mathbb{Z}}}

\newcommand{\Bx}{{\mathbf{x}}}
\newcommand{\By}{{\mathbf{y}}}

  \newcommand{\A}{{\mathcal{A}}}
  
  \newcommand{\C}{{\mathcal{C}}}
  
  \newcommand{\E}{{\mathcal{E}}}
  \newcommand{\F}{{\mathcal{F}}}

  \newcommand{\I}{{\mathcal{I}}}
  \newcommand{\J}{{\mathcal{J}}}
    
\renewcommand{\L}{{\mathcal{L}}}
  \newcommand{\M}{{\mathcal{M}}}
  \newcommand{\N}{{\mathcal{N}}}
\renewcommand{\O}{{\mathcal{O}}}

  \newcommand{\Q}{{\mathcal{Q}}}
  
\renewcommand{\S}{{\mathcal{S}}}
  
  \newcommand{\UUU}{{\mathcal{U}}}
  \newcommand{\V}{{\mathcal{V}}}
  \newcommand{\W}{{\mathcal{W}}}

\newcommand{\AND}{\text{ and }}

\newcommand{\fp}{\mathfrak{p}}

\newcommand{\fa}{\mathfrak{a}}

\newcommand{\fc}{\mathfrak{c}}
\newcommand{\fm}{\mathfrak{m}}

\newcommand{\fh}{\mathfrak{h}}

\newcommand{\fN}{\mathfrak{N}}

\newcommand{\Pic}{\operatorname{Pic}}

\newcommand{\Gal}{\operatorname{Gal}}

\newcommand{\GL}{\operatorname{GL}}

\newcommand{\Aut}{\operatorname{Aut}}

\newcommand{\fC}{\mathfrak{C}}
\newcommand{\fF}{\mathfrak{F}}

\newcommand{\fD}{\mathfrak{D}}

\newcommand{\ep}{\varepsilon}

\newcommand{\ol}{\overline}

\newcommand{\upchi}{{\raise.35ex\hbox{$\chi$}}}

\newcommand{\SL}{\operatorname{SL}}
\newcommand{\Vol}{\operatorname{Vol}}

\makeatletter
\@namedef{subjclassname@2010}{%
  \textup{2010} Mathematics Subject Classification}
\makeatother

\newtheorem{theorem}{Theorem}[section]
\newtheorem{corollary}[theorem]{Corollary}
\newtheorem{proposition}[theorem]{Proposition}
\newtheorem{lemma}[theorem]{Lemma}
\newtheorem{conjecture}[theorem]{Conjecture}

\theoremstyle{definition}
\newtheorem{definition}[theorem]{Definition}

\newtheorem{remark}[theorem]{Remark}

\numberwithin{equation}{section}

\begin{document}

\title{Binary quartic forms with vanishing $J$-invariant}

\author{Stanley Yao Xiao}
\address{Department of Mathematics \\
University of Toronto \\
Bahen Centre \\
40 St. George Street, Room 6290 \\
Toronto, Ontario, Canada \\  M5S 2E4 }
\email{syxiao@math.toronto.edu}\indent


\begin{abstract} We obtain an asymptotic formula for the number of $\GL_2(\bZ)$-equivalence classes of  irreducible, totally real binary quartic forms with integer coefficients with vanishing $J$-invariant. These results give a case where one is able to count integral orbits inside a relatively open real orbit of a variety of degree at least three which is closed under a group action. As a consequence, we give an asymptotic formula for the number of $\GL_2(\bZ)$-classes of irreducible binary quartic forms with vanishing $J$-invariant and Galois group $C_4$, ordered by discriminant. Our method of proof introduces a new observation regarding taking powers in the class group of quadratic forms of a given discriminant and lattices associated to representable primes (see appendix by Erick Knight).
\end{abstract}

\maketitle

\section{Introduction}
\label{Intro}

Let 
\begin{equation} \label{quartic form} F(x,y) = a_4 x^4 + a_3 x^3 y + a_2 x^2 y^2 + a_1 xy^3 + a_0 y^4 \in \bR[x,y]
\end{equation}
be a binary quartic form, and put $V_4(\bR)$ for the 5-dimensional vector space of real binary quartic forms. The group $\GL_2(\bR)$ acts on $V_4(\bR)$ via the  \emph{substitution action}, defined for $T = \left(\begin{smallmatrix} t_1 & t_2 \\ t_3 & t_4 \end{smallmatrix} \right) \in \GL_2(\bR)$ and $F \in V_4(\bR)$ by
\begin{equation} \label{twisted} F_T(x,y) = F(t_1 x + t_2 y, t_3 x + t_4 y).
\end{equation}
It is well-known that the ring of relative polynomial invariants of the substitution action of $\GL_2(\bR)$ on $V_4(\bR)$ is a polynomial ring generated by two elements, commonly denoted as $I$ and $J$. They are given by
\begin{equation} \label{I} I(F) = 12 a_4 a_0 - 3 a_3 a_1 + a_2^2 \end{equation}
and
\begin{equation} \label{J} J(F) = 72 a_4 a_2 a_0 + 9 a_3 a_2 a_1 - 27 a_4 a_1^2 - 27 a_0 a_3^2 - 2 a_2^3. \end{equation}
Both the quadric defined by $I(F) = 0$ and the cubic defined by $J(F) = 0$ are invariant under $\GL_2(\bR)$; that is, for all $F \in V_\bR$ and $T \in \GL_2(\bR)$, we have $J(F) = 0$ if and only if $J(F_T) = 0$. \\

Put
\begin{equation} \label{J = 0} \V_4(\bR) = \{F \in V_4(\bR) : J(F) = 0\} \end{equation}
for the real cubic threefold defined by the vanishing of $J$ in $V_4(\bR)$. $\V_4(\bR) \setminus \{F : \Delta(F) = 0\}$ consists of three relatively open orbits under the substitution action of $\GL_2(\bR)$, consisting of non-singular forms with 0, 2, or 4 real linear factors respectively. We shall denote by $\V_4^{(i)}(\bR)$ the orbit of $\V_4(\bR)$ consisting of forms with $4-2i$ real linear factors. In particular, 
\[\V_4^{(0)}(\bR) = \{F \in \V_4(\bR) : F \text{ has 4 real linear factors}\}.\]
The \emph{discriminant} $\Delta(F)$ of a binary quartic form $F$ is expressible in terms of $I$ and $J$ as
\begin{equation} \label{disc F} \Delta(F) = \frac{4I(F)^3 - J(F)^2}{27}. \end{equation}


Put
\[W_n(\bZ) = \{\GL_2(\bZ)\text{-orbits of integral binary }n\text{-ic forms}\}.\]
Since $I(F), J(F)$ are $\GL_2(\bZ)$-invariants, for any class $w \in W_4(\bZ)$ and any $F,G \in w$ we have $I(F) = I(G)$ and $J(F) = J(G)$. Thus, the values of $I,J$ are well-defined on the class $w$. Now put
\begin{equation} \W_4(\bZ) = \{w \in W_4(\bZ) : J(w) = 0\}.\end{equation}

The main goal of this paper is to establish asymptotic formulae for two subclasses of $\W_4(\bZ)$ with non-zero discriminant, where we count the orbits by discriminant. This count is a priori finite by a well-known result of Borel and Harish-Chandra \cite{BHC}. Since the number of real linear factors are preserved under $\GL_2(\bR)$ action, one can define the number of real linear factors for orbits in $\W_4(\bR)$. Indeed, we shall put 
\[\W_4^{(i)}(\bZ) = \{w \in \W_4(\bZ): F \in \V_4^{(i)}(\bR) \text{ for all } F \in w\}.\]

The first family we shall consider is $\W_4^{(0)}(\bZ)$. For a positive number $X$, put
\begin{equation} \label{NX} N(X) = \#\{w \in \W_4^{(0)} (\bZ) : \Delta(w) \ne 0, \Delta(w) \leq X\} .\end{equation} 
We denote the set on the right hand side above by $\W_4^{(0, \dagger)}(X)$.  Further note that for all binary quartic forms $F$ with real coefficients and 4 real linear factors,  we have $\Delta(F) > 0$. \\

The forms with vanishing $J$-invariant can be characterized by the fact that their \emph{Hessian} covariants are perfect squares in $\bC[x,y]$. The Hessian covariant of $F$, denoted as $H_F$, is given by
\begin{equation} \label{Hess} H_F(x,y) = (3a_3^2 - 8a_4 a_2)x^4 + 4(a_3 a_2 - 6a_4 a_1)x^3y + 2(2a_2^2 - 24a_4 a_0 - 3a_3 a_1)x^2y^2 \end{equation}
\[ + 4(a_2 a_1 - 6 a_3 a_0)xy^3 + (3a_1^2 - 8 a_2 a_0)y^4.\]
Then $F \in \V_4(\bR)$ if and only if there exists a quadratic form $f$ with complex coefficients such that $f^2 | H_F$ as elements in $\bC[x,y]$. Further, one can take $f$ to be a form with co-prime integer coefficients and non-zero discriminant when $F \in \V_4(\bZ)$ and $I(F) \ne 0$. Moreover when $F \in \V_4^{(0)}(\bR)$ we may take $f$ to have real coefficients and $\Delta(f) < 0$ (see Lemma \ref{explicit I}). Observe that if $H_F$ is divisible by the square of a reducible quadratic form, then so will $H_{F_T}$ for any $T \in \GL_2(\bZ)$. The next family we shall consider will be:
\[\W_4^\star(\bZ) = \{w \in \W_4(\bZ) : \text{for all } F \in w, H_F \text{ is divisible by the square of a reducible } \]
\[\text{quadratic form } f\}.\]
We now put 
\[M(X) = \#\{w \in \W_4^\star(\bZ) : |\Delta(w)| \leq X, w \text{ is irreducible}.\}\]

The main theorems of our paper are the following counting results for $N(X)$ and $M(X)$:

\begin{theorem} \label{MT} The asymptotic formula
\[N(X) = \frac{6\sqrt[3]{2}  \zeta(2) }{7 \zeta(3)} X^{1/3} \log X + O \left(X^{1/3} \right). \]
holds. 
\end{theorem} 

\begin{theorem} \label{MT2} The asymptotic formula
\[M(X) = \frac{\zeta(2) }{6 \sqrt[3]{4} \zeta(3)} X^{1/3} \log X  + O\left(X^{1/3} \right)    \]
holds.
\end{theorem}


We note that the methods introduced in this paper just barely fall short of being able to give the analogous result in Theorem \ref{MT} for $\W_4^{(2)}(\bZ)$ and $\W_4^{(4)}(\bZ)$. This lacuna will be filled in in a future paper. \\

Since by definition elements in $\W_4^{(0)}$ we have $J(w) = 0$, and thus for any such orbit and any irreducible $F \in w$, we have that the Galois group of the splitting field of $F$ is a subgroup of $D_4$ (see \cite{TX} for a full treatment). We have the following: 

\begin{theorem} \label{Gal thm} Let $f$ be a positive definite, primitive integral binary quadratic form. Let $F \in \V_f(\bZ)$ be an irreducible binary quartic form. Then $\Gal(F) \cong C_4$ if and only if $-\Delta(f)$ is a square. 
\end{theorem}

The proof of Theorem \ref{Gal thm} relies on the fact that whenever $-\Delta(f) = \square$, any form $F \in \V_f(\bZ)$ with square discriminant is \emph{necessarily reducible}. Thus, following the criteria determining all possible Galois groups of quartic forms in \cite{Con}, all irreducible elements must necessarily have Galois group $C_4$. This fact is based on the existence of a natural involution on $\V_4^{(0)}(\bR)$. \\

Let $F \in V_\bR$ be as given in (\ref{quartic form}). It has a natural \emph{sextic covariant} given by the Jacobian determinant of $F$ and the Hessian covariant $H_F$ of $F$. It has the following explicit formula: 
\begin{align} \label{sextic cov} 
F_6(x,y) & = (a_3^3 + 8a_4^2 a_1 - 4a_4 a_3 a_2) x^6 + 2(16a_4^2 a_0 + 2 a_4 a_3 a_1 - 4a_4 a_2^2 + a_3^2 a_2)x^5y \\
& + 5(8a_4 a_3 a_0 + a_3^2 a_1 - 4a_4 a_2 a_1)x^4y^2 + 20(a_3^2 a_0 - a_4 a_1^2) x^3 y^3 \notag \\
& - 5(8a_4 a_1 a_0 + a_3 a_1^2 - 4 a_3 a_2 a_0)x^2 y^4 - 2(16a_4 a_0^2 + 2 a_3 a_1 a_0 - 4 a_2^2 a_0 + a_2 a_1^2)xy^5   \notag \\ 
& - (a_1^3 + 8a_3 a_0^2 - 4 a_2 a_1 a_0)y^6. \notag 
\end{align}
In \cite{X} we proved that $F_6$ is always a \emph{Klein form} (see \cite{BenSan}). Moreover, when $F \in \bR[x,y]$ and $J(F) = 0$ it admits a factorization of the shape
\[F_6(x,y) = f(x,y) G_F(x,y),\]
where $f(x,y)$ is a binary quadratic form with real coefficients such that $f^2 | H_F$ over $\bC$ and $J(G_F) = 0$. Recall that when $F$ is totally real we will see that $f$ has negative discriminant. We now choose $f = f_F$ such that $\Delta(f) = -4$, and define $G_F$ as the quotient $F_6/f \in \bR[x,y]$. We then have the following:
\begin{theorem} \label{V0 inv} The map $\Xi: \V_4^{(0)}(\bR) \rightarrow V_\bR$ defined by
\[\Xi(F) = G_F\]
satisfies $\Xi^2(F) = \alpha_F F$ for some $\alpha_F \in \bR$. 
\end{theorem}
Of course the involution $\Xi$ need not restrict to an involution from $\V_4^{(0)}(\bZ)$ to $\V_4^{(0)}(\bZ)$. However we will see that there is a rational version $\Xi_\bQ$ of $\Xi$ such that $\Xi_\bQ^2(F) = c F$, for some $c \in \bZ$. This is enough to control the reducible forms $F \in \V_f(\bZ)$ with $\Delta(F) = \square$ or $-\Delta(F) = \square$. \\

Theorem \ref{Gal thm} and an easier case of Theorem \ref{MT} have the following attractive consequence. For $w \in W_4(\bZ)$ define the \emph{Bhargava-Shankar height} to be
\begin{equation} \label{BS height} H_{\text{BS}}(w) = \max\{|I(w)|^3, J(w)^2/4\}.
\end{equation} 
For a transitive subgroup $G$ of the symmetric group $S_4$, put 
\[\N_G(X) = \# \{w \in W_4(\bZ) : H_{\text{BS}}(w) \leq X, \Gal(w) \cong G\}\]
and
\[\M_G(X) = \#\{w \in W_4(\bZ) : |\Delta(w)| \leq X, \Gal(w) \cong G\}.\]
Let $N_{G}(X), M_G(X)$ denote respectively the sub-count of $\N_G(X), \M_G(X)$ restricted to orbits with $J  = 0$. We obtain the theorem:
\begin{theorem} \label{C4 MT} Let $\ep > 0$. We have the asymptotic formulae 
\[N_{C_4}(X) = \frac{7}{9} X^{1/3} + O_\ep \left(X^{1/3 - \ep} \right)\]
and
\[M_{C_4}(X) = \frac{7}{6 \sqrt[3]{2}} X^{1/3} + O_\ep \left(X^{1/3 - \ep}  \right).\]
\end{theorem}

Theorem \ref{C4 MT} then implies:
\begin{corollary} \label{C4 mag} There exist positive numbers $c_0, c_1, c_2$ such that for any $X > c_0$ we have
\[\N_{C_4}(X) > c_1 X^{1/3}\]
and
\[\M_{C_4}(X) > c_2 X^{1/3}.\]
\end{corollary}

In \cite{TX} we proved, along with Tsang, that the number of $C_4$-forms with a fixed Cremona covariant and bounded Bhargava-Shankar height $X$ is $O_{f,\ep} \left(X^{1/6 + \ep}\right)$. In fact the forms corresponding to $C_4$-forms in a fixed family parametrized by the quadratic form $f$ are the rational points on a certain toric, singular del Pezzo surface of degree 4. This leads us to conjecture the following:

\begin{conjecture} Let $\ep > 0$. Then 
\[\M_{C_4}(X) = M_{C_4}(X) + O_\ep \left(X^{1/6 + \ep}\right) \text{ and } \N_{C_4}(X) = N_{C_4}(X) + O_\ep \left(X^{1/6 + \ep}\right).\]
\end{conjecture}

In \cite{BhaSha}, Bhargava and Shankar proved that
\[\N_{S_4}(X) = \frac{44 \zeta(2)}{135} X^{5/6} + O_\ep \left(X^{3/4 + \ep}\right),\]
which they used to obtain their magnificent theorem on the boundedness of average Mordell-Weil rank of elliptic curves over $\bQ$. It remains a significant challenge to estimate $\M_{S_4}(X)$ from above.\\

In \cite{TX} and \cite{TX2} Tsang and I proved that
\begin{equation} \label{ND4 mag} \N_{D_4}(X) \gg X^{1/2} \log X
\end{equation}
and
\begin{equation} \label{MD4 mag}  \M_{D_4}(X) \gg X^{1/2} (\log X)^2.
\end{equation}
We also showed that 
\begin{equation} \label{V4 mag} \N_{V_4}(X), \M_{V_4}(X) \gg X^{1/3}.\end{equation}
We expect that both (\ref{ND4 mag}) and (\ref{MD4 mag}) represent the true orders of magnitude. Comparing (\ref{V4 mag}) and Corollary \ref{C4 mag}, one has to wonder whether $V_4$-forms or $C_4$-forms are more numerous.\\  

It remains a difficult challenge to estimate $\N_{A_4}(X), \M_{A_4}(X)$. \\

Theorem \ref{MT} represents the first case where one can count integral $G(\bZ)$-orbits inside a relatively open orbit $G(\bR) \cdot v$, where $v \in V(\bR)$ sits inside a \emph{proper subvariety} of degree at least three which is closed under the action of $G(\bR)$. In our case, the group $G$ is $\GL_2(\bR)$ and the variety is the cubic threefold in $V_4(\bR)$ given by $J(F) = 0$. The methods we employ in this paper, while heavily inspired by the work of Bhargava, do not directly involve his geometry of numbers method and the action of $\GL_2(\bR)$ on $\V_4(\bR)$ is not directly exploited. Instead, we partition $\V_4(\bR)$ into families indexed by $\GL_2(\bZ)$-equivalence classes of integral binary quadratic forms, as we did in \cite{TX}. This reduces the problem of counting integral orbits in $\V_4(\bR)$ to counting integer points, sorted by discriminant, inside a countable collection of 2-dimensional vector spaces inside $\V_4(\bR)$. We then use a wide assortment of results regarding binary quadratic forms to help establish Theorem \ref{MT}. Of particular note is Proposition \ref{hensel lift prop}, which is a novel observation regarding the change in the $\SL_2(\bZ)$-class of quadratic forms as one performs `Hensel lifting' of lattices containing primitive solutions to the congruence $f(x,y) \equiv 0 \pmod{p^k}$. \\

We remark that S.~Ruth, in his thesis, gave an argument that essentially counts the number of $\GL_2(\bZ)$-orbits of quartic Klein forms (see \cite{BenSan} for a modern treatment), that is, those quartic forms with $I(F) = 0$. However there is some doubt that his application of Heath-Brown's circle method \cite{HB} is acceptable. In January 2019 the author heard a lecture given by A.~Alpoge on this matter, which resolved the issue by giving an argument which circumvents the problematic application of Heath-Brown's method. Therefore, the asymptotic formula given for the number of quartic Klein forms of bounded discriminant given by Ruth is correct. \\

Theorem \ref{MT2}, in comparison, is relatively straightforward. This is mostly because the class number of reducible quadratic forms is very easy to understand, and that the set of discriminants of reducible quadratic forms is equal to the set of square integers, which is a very thin set. We give the proof of Theorem \ref{MT2} in Section \ref{MT2 proof}. \\

In view of Theorems \ref{MT} and \ref{MT2}, all that is needed to prove the full asymptotic formula for the number of $\GL_2(\bZ)$-orbits of binary quartic forms $F$ with vanishing $J$-invariant is to count the number of integral orbits whose Hessians are divisible by the square of an \emph{irreducible, indefinite} binary quadratic form. There are significant barriers to carrying out the arguments in this paper to handle this case, but there is another method to count such orbits. We wish to expand on this in future work. \\

Finally, Proposition \ref{hensel lift prop} appears to be a new observation regarding taking powers in the class group of quadratic forms of a given negative discriminant and may be of separate interest. The author thanks Erick Knight for providing the proof in the appendix which is much more elegant than his own. 

\subsection*{Acknowledgements} We thank an anonymous referee who identified a key error in a previous version of this paper. We thank J.~Friedlander for suggesting the paper \cite{Schi}, which contains the key ideas necessary to carry out the proof of Theorem \ref{MT}. We thank A.~Alpoge for his lecture given at the Joint Mathematics Meetings in Baltimore, which settled the question of whether the correct asymptotic formula for the number of $\GL_2(\bZ)$-orbits of quartic Klein forms has been obtained.

\section{Parametrizing quartic forms with $J = 0$ by the Hessian}
\label{parametrization} 

In this section, we will refine our parametrization theorem in our work with Tsang in \cite{TX} to provide a parametrization theorem for binary quartic forms with vanishing $J$-invariant. For a binary quadratic form $f$ with integer coefficients, put $\C(f)$ for its $\GL_2(\bZ)$-equivalence class, and put
\[\V_f(\bR) = \{F \in \V_4(\bR) : f^2 | H_F\}, \]
and $\V_f(\bZ)$ for the subset of $\V_f(\bR)$ consisting of those forms with integer coefficients. We then put
\[\W_f(\bZ) = \{w \in \W_4(\bZ) : \exists F \in w \text{ s.t. } f^2 | H_F\}\]
and $\W_f^\dagger(\bZ)$ to be the subset consisting of irreducible elements. Notice that if $f$ and $g$ are $\GL_2(\bZ)$-equivalent, then $\W_f(\bZ) = \W_g(\bZ)$; hence $\W_f(\bZ)$ only depends on $\C(f)$, so we write $\W_{\C(f)}(\bZ)$ instead. We then have the following result:

\begin{proposition} \label{families} We have that $\W_4(\bZ)$ is given by the disjoint union
\[\W_4(\bZ) = \bigcup_{\C(f)} \W_{\C(f)}(\bZ),\]
where $\C(f)$ varies over all $\GL_2(\bZ)$-equivalence classes of primitive, integral binary quadratic forms with non-zero discriminant. 
\end{proposition}

Our goal is to obtain, for each $\C(f)$, a set of representatives in $\V_f(\bZ)$ for $\W_{\C(f)}(\bZ)$. We begin with the following lemma:

\begin{lemma} \label{para 1 distinct} Let $f,g$ be two binary quadratic forms with co-prime integer coefficients and let $F,G \in \V_\bZ$ be such that $F \in \V_{f,\bZ}, G \in \V_{g,\bZ}$. If $F$ and $G$ are $\GL_2(\bZ)$-equivalent, then $f$ is $\GL_2(\bZ)$-equivalent to either $g$ or $-g$. 
\end{lemma}

\begin{proof} This follows from the fact that if $F,G$ are $\GL_2(\bZ)$-equivalent then their Hessians $H_F, H_G$ are $\GL_2(\bZ)$-equivalent, since $H_F$ is a covariant of $F$. This implies that $f^2$ is $\GL_2(\bZ)$-equivalent to $g^2$. Taking square roots shows that $f$ is $\GL_2(\bZ)$-equivalent to $\pm g$.  \end{proof}

We show that $\V_{f}(\bZ)$ has a natural structure as a 2-dimensional lattice, and we give an explicit embedding of $\V_{f}(\bZ)$ into $\bZ^2$ below. First note that from (\ref{J}) we see that $a_2 \equiv 0 \pmod{3}$. Put
\begin{equation} \label{f-alpha} f(x,y) = \alpha x^2 + \beta xy + \gamma y^2, \alpha, \beta, \gamma \in \bZ, \end{equation}
with
\begin{equation} \label{cong 1} \A_1 = 4\gamma A - \beta B,\end{equation}
\begin{equation} \label{cong 2} \A_2 = 4 \beta \gamma A - (\beta^2 - \alpha \gamma)B,\end{equation}
and
\begin{equation} \label{cong 3} \A_3 = 4 \gamma (\beta^2 - \alpha \gamma) A - \beta (\beta^2 - 2 \alpha \gamma)B.\end{equation}
Define the lattice $\L_{f, \alpha}$ by 
\begin{equation} \label{Lfa} \L_{f,\alpha} = \{(A,B) \in\bZ^2 : \A_1 \equiv 0 \pmod{2 \alpha}, \A_2 \equiv 0 \pmod{\alpha^2}, \A_3 \equiv 0 \pmod{4 \alpha^3} \}.\end{equation}
We have the following result regarding the family $\V_{f,\bZ}$:

\begin{proposition} \label{para 1} Let $f(x,y) = \alpha x^2 + \beta xy + \gamma y^2, \alpha \ne 0$ be a binary quadratic form with co-prime integer coefficients and non-zero discriminant. The forms in $\V_{f}(\bZ)$ are of the form 
\begin{equation}\label{J family}
\left\lbrace
\begin{array}{@{}c@{}c}
Ax^4 + Bx^3y - \dfrac{12\gamma A - 3 \beta B}{2 \alpha} x^2 y^2 + \left(\dfrac{-4 \beta \gamma A + (\beta^2 -  \alpha \gamma) B}{\alpha^2} \right)xy^3 \\\\
+\left(\dfrac{-4\gamma(\beta^2 - \alpha \gamma) A + \beta(\beta^2 - 2 \alpha \gamma) B}{4\alpha^3}\right)y^4:
(A,B) \in \L_{f,\alpha}
\end{array}
\right\rbrace.
\end{equation}
In particular, the map $\nu : \V_{f}(\bZ) \rightarrow \L_{f,\alpha}$ given by $\nu(Ax^4 + Bx^3 y + \cdots ) = (A,B)$ is a bijection between $\V_{f}(\bZ)$ and $\L_{f,\alpha}$. 
\end{proposition}
To prove Proposition \ref{para 1}, we shall require some results from \cite{X} and \cite{TX} and recall some relevant notions. For a given binary quartic form $F$ with real coefficients, define the \emph{automorphism group} of $F$ (over $\bR$) as
\begin{equation} \label{auto} \Aut_\bR F = \{T \in \GL_2(\bR) : F_T(x,y) = F(x,y)\}.\end{equation}
For a subring $\bF$ of $\bR$, define 
\[\Aut_\bF F = \{T \in \Aut_\bR F: \exists \lambda \in \bR \text{ s.t. } \lambda T \in \GL_2(\bF)\}.\]
For a given binary quadratic form $f(x,y) = \alpha x^2 + \beta xy + \gamma y^2$ with real coefficients, define
\begin{equation} \label{Mf} M_f = \frac{1}{\sqrt{|\Delta(f)|}} \begin{pmatrix} \beta & 2 \gamma \\ -2 \alpha & - \beta \end{pmatrix}.\end{equation}
Put 
\[V_{f}(\bR) = \{F \in V_4(\bR) : M_f \in \Aut_\bR F\}.\]

We can now prove the following, which identifies $\V_f(\bR)$ as a plane inside $V_f(\bR)$:

\begin{lemma} \label{plane} Let $f$ be a binary quadratic form with real coefficients and non-zero discriminant. Then $\V_f(\bR)$ is the plane inside $V_f(\bR)$ defined by 
\[12 \gamma A - 3 \beta B + 2 \alpha C = 0.\]
\end{lemma}

\begin{proof} By the results in \cite{TX}, it follows that for any $F \in V_f(\bR)$ we have $f^2$ divides
\[\fF_1(x,y) = \frac{1}{3} \left(H_F(x,y) + 4L_f(F) F(x,y)\right),\]
where 
\[L_f(F) = - \frac{12 \gamma A - 3 \beta B + 2 \alpha C}{2\alpha}.\]
It therefore follows that $F \in \V_f(\bR)$ if and only if $L_f(F) = 0$ or $F(x,y)$ is proportional to $H_F$. The latter implies that $\Delta(F) = 0$, so the former must hold. 
\end{proof} 

\begin{proof}[Proof of Proposition \ref{para 1}] By (3.1) of \cite{TX}, it follows that whenever $F \in V_f(\bR)$, the $xy^3, y^4$ coefficients of $F$ are given by linear forms in the coefficients of $x^4, x^3 y, x^2 y^2$. In particular, we have
\[F(x,y) = Ax^4 + Bx^3 y + Cx^2 y^2 +  + \left(\dfrac{4 \beta \gamma A - (\beta^2 + 2 \alpha \gamma) B + 2 \alpha \beta C}{2\alpha^2} \right)xy^3 \]
\[ +\left(\dfrac{4 \gamma(\beta^2 + 2\alpha \gamma) A - \beta(\beta^2 + 4 \alpha \gamma) B + 2\alpha \beta^2 C}{8\alpha^3}\right)y^4\]
for $A,B,C \in \bR$. Moreover we see that $f^2$ is proportional to the quartic form
\[\frac{1}{3} \left(H_F(x,y) + 4 L_f(F) F(x,y) \right).\] 
The condition that $f^2 | H_F$ implies that $L_f(F) = 0$, or equivalently, 
\[C = \frac{-12 \gamma A + 3 \beta B}{2 \alpha}.\]
We then see that the condition $F \in V_\bZ$ is then equivalent to $(A,B) \in \L_{f,\alpha}$, as desired. \end{proof}

Our aim now is to show that $\V_f(\bZ)$ is an $n$-cover for $\W_{\C(f)}(\bZ)$, where $n$ is a positive integer which is absolutely bounded. Typically we will have $n = 1$. We shall precisely describe when $\V_f(\bZ)$ fails to be in one-to-one correspondence with $\W_{\C(f)}(\bZ)$. We shall need the following definition:

\begin{definition} Let $f$ be a binary quadratic form with integer coefficients and non-zero discriminant. Then $f$ is said to be \emph{ambiguous} if there exists a $\GL_2(\bZ)$-translate $g = g_2 x^2 + g_1 xy + g_0 y^2$ of $f$ such that $g_2 | g_1$. We say that $f$ is \emph{opaque} if there exists a $\GL_2(\bZ)$-translate $g$ of $f$ which takes the shape $g(x,y) = g_2 x^2 + g_1 xy - g_2 y^2$. 
\end{definition}

We summarize some results of \cite{TX} as follows:

\begin{proposition} \label{bad class} Let $f = \alpha x^2 + \beta xy + \gamma y^2$ be a primitive integral binary quadratic form with non-zero discriminant. Then there exists a positive integer $n_f$ such that $\V_f(\bZ)$ is a $n_f$-fold cover of $\W_{\C(f)}(\bZ)$, where
\[ n_f = \begin{cases} 1 &\text{if } f \text{ is neither ambiguous nor opaque;} \\ 
6 & \text{if } f \text{ is } \GL_2(\bZ)\text{-equivalent to } x^2 + xy + y^2; \\
4 & \text{if } f \text{ is ambiguous and opaque}; \\
2 & \text{otherwise}.
\end{cases}
\]
\end{proposition}

We note that if a quadratic form $f$ is opaque, then its discriminant is positive; hence no positive definite binary quadratic form $f$ is opaque. \\

When $f$ is positive definite, then the number of elements in $\V_{f}(\bZ)$ of bounded height is finite and in fact lie in an ellipse. We shall enumerate these elements in Section \ref{total real}.

\section{Outer parametrization of binary quartic forms with $J = 0$}
\label{outer para}

In \cite{TX2}, we obtained a new parametrization of binary quartic forms with small Galois groups. Let $h$ be an integral binary quadratic form given as  
\[h(x,y) = h_2 x^2 + h_1 xy + h_0 y^2,\]
and analogous expressions for $u(x,y)$ and $v(x,y)$. Next put $\J(u,v)$ for the \emph{Jacobian determinant} of $u$ and $v$, given by
\begin{equation} \label{Jac det} \J(u,v)(x,y) = \frac{1}{2} \begin{vmatrix} \frac{\partial u}{\partial x} & \frac{\partial u}{\partial y} \\ \frac{\partial v}{\partial x} & \frac{\partial v}{\partial y} \end{vmatrix} = (u_2 v_1 - u_1 v_2)x^2 + 2(u_2 v_0 - u_0 v_2)xy + (u_1 v_0 - u_0 v_1)y^2. 
\end{equation}
We say that a pair of integral binary quadratic forms $(u,v)$ is \emph{primitive} if the $\gcd$ of the coefficients of $\J(u,v)$ is at most 2. \\ \\
For a binary quartic form $F$, define its \emph{cubic resolvent} to be the cubic polynomial
\begin{equation} \label{cubic resolvent} \Q_F(x) = x^3 - 3I(F)x + J(F). \end{equation}
We obtained the following in \cite{TX2}:

\begin{proposition}[Tsang, Xiao] Let $F$ be a binary quartic with integer coefficients and non-zero discriminant. Then $\Q_F(x)$ is reducible over $\bQ$ if and only if there exists an integral binary quadratic form $h$ and a pair of primitive integral binary quadratic forms $u,v$ such that
\begin{equation} \label{outer para eq} F(x,y) = h(u(x,y), v(x,y)). \end{equation}
\end{proposition}

We further see that $J(F) = 0$ with $\J(u,v)^2 | H_F$ if and only if 
\begin{equation} \label{outer L} \Delta(v) h_0 - \Delta(u,v) h_1 + \Delta(u) h_2 = 0, \end{equation}
where
\begin{equation} \label{joint disc} \Delta(u,v) = 2 u_2 v_0 - u_1 v_1 + 2 u_0 v_2
\end{equation} 
is the \emph{joint discriminant} of the pair $(u,v)$ of binary quadratic forms. We now give a brief overview of the invariant theory of \emph{pairs} of binary quadratic forms.

\subsection{The action of $\GL_2(\bR) \times \GL_2(\bR)$ on pairs of binary quadratic forms} We shall denote by $\UUU_{2,2}(\bR)$ to be the six-dimensional $\bR$-vector space of pairs of binary quadratic forms. That is, 
\begin{equation} \label{pairs of quadratic} \UUU_{2,2}(\bR) = \left\{\left(\begin{pmatrix} f_2 & f_1/2 \\ f_1/2 & f_0 \end{pmatrix}, \begin{pmatrix} g_2 & g_1/2 \\ g_1/2 & g_0 \end{pmatrix} \right) : f_2, f_1, f_0, g_2, g_1, g_0 \in \bR\right\}.
\end{equation}
The group $G(\bR) = \GL_2(\bR) \times \GL_2(\bR)$ acts on $\UUU_{2,2}(\bR)$ as follows. For $T = \left(\begin{smallmatrix} t_1 & t_2 \\ t_3 & t_4 \end{smallmatrix}\right)$ and $S \in \GL_2(\bR)$, with $S^t$ denoting the transpose of $S$, we have
\[(T, S) \star (A,B) = \left(t_1 SAS^t + t_2 SBS^t, t_3 SAS^t + t_4 SBS^t \right).\] 
The actions of the two copies of $\GL_2(\bR)$ commute, and we refer to the action of the first copy of $\GL_2(\bR)$ the \emph{outer action} and the action of the second copy the \emph{inner action}. \\

We now have the following:
\begin{lemma} Let $T \in \GL_2(\bR)$ be such that $\det T = \pm 1$. Put $(U,V) = (T, I_{2 \times 2}) \star (u,v)$. Then $\J(U,V) = \J(u,v)$. 
\end{lemma}

\begin{proof} Simple numerical verification. 
\end{proof}

We next define the \emph{invariant quadratic form} of a pair of quadratic forms $(f,g)$, which is given as
\begin{align*} \F(x,y) & = \F_{(u,v)}(x,y) =  -\det \left(\begin{pmatrix} 2 u_2 & u_1 \\ u_1 & 2 u_0 \end{pmatrix} x - \begin{pmatrix} 2 v_2 & v_1 \\ v_1 & 2 v_0 \end{pmatrix} y\right) \\
& = \Delta(u)x + 2 \Delta(u,v) xy + \Delta(v) y^2,\end{align*}
We then have the following lemma:
\begin{lemma} The polynomials $\Delta(u), \Delta(u,v), \Delta(v)$ are the generators of the ring of polynomial invariants of the inner action of $\GL_2(\bR)$ on the set of pairs of binary quadratic forms. In particular, $\F(x,y)$ is invariant with respect to inner action.
\end{lemma}
This result is classical; see for example \cite{Mor}. See also \cite{HCL} for a modern view. \\ 

A simple calculation reveals the following:
\begin{lemma} \label{double reduced} Let $(u,v)$ be a pair of binary quadratic forms. Then
\[\Delta(\F) = 4 \Delta(\J(u,v)).\]
\end{lemma}
Furthermore, it is easy to check that the outer action on the pair $(f,g)$ induces the usual substitution action of $\GL_2(\bR)$ on $\F(x,y)$. \\ \\
Our goal is to show that when $f = \J(u,v)$ is positive definite (hence, the invariant quadratic form $\F(x,y)$ is necessarily positive definite by Lemma \ref{double reduced}), there is essentially a \emph{unique} choice of a pair of quadratic forms $(u,v)$ such that (\ref{outer para eq}) holds and both $\J(u,v)$ and $\F(x,y)$ are reduced. We shall prove the following:
\begin{proposition} \label{outer can} Let $F$ be a binary quartic form with integer coefficients, non-zero discriminant, and $J(F) = 0$. Then there exists a primitive pair of binary quadratic forms $(u,v)$ such that $\J(u,v)$ and $\F_{(u,v)}$ are both reduced with positive leading coefficients and integers $h_2, h_1, h_0$ such that (\ref{outer para eq}) holds. Moreover, the pair $(u,v)$ is uniquely determined up to the action of $\Aut_\bZ(\F_{(u,v)}) \times \Aut_\bZ (\J(u,v)) \subset G(\bZ)$. 
\end{proposition}

\begin{proof} Given any pair $(u,v)$ for which (\ref{outer para eq}) holds, the outer action of $\GL_2(\bZ)$ induces a change of variables of the quadratic form $h$; so any outer translate produces another representation of the shape (\ref{outer para eq}). Similarly, inner action preserves the representability of $F$ in the shape (\ref{outer para eq}). Now suppose that both $\J(u,v)$ and $\F_{(u,v)}$ are fixed. Then the outer action is restricted to the subset of $\GL_2(\bZ)$ which fixes $\F_{(u,v)}$, in other words, $\Aut_\bZ(\F)$. Similarly, inner action is restricted to $\Aut_\bZ(\J(u,v))$. 
\end{proof}

We make the trivial observation that for any $A = \left(\begin{smallmatrix} a & b \\ c & d \end{smallmatrix} \right) \in \GL_2(\bZ)$, any binary quadratic form $h$ and any pair of quadratic forms $(u,v)$, we have
\begin{equation} \label{outer stab act} h(u,v) = h_{M^{-1}}(au + bv, cu + vd).
\end{equation}

\section{$\SL_2(\bZ)$-classes of binary quadratic forms and the Picard group of quadratic orders}
\label{quadratic stuff}

Since $\V_f(\bZ)$ is canonically isomorphic to $\V_g(\bZ)$ whenever $f$ and $g$ are $\GL_2(\bZ)$-equivalent, it is thus pertinent to examine the properties of $\GL_2(\bZ)$-equivalence classes of binary quadratic forms. There is a rich history to this subject, and we will only pick from it what we need for the present work. See \cite{HCL} for a modern treatment. \\ 

For technical reasons, we shall deal with $\SL_2(\bZ)$-equivalence classes of binary quadratic forms. For a binary quadratic form $f$, denote by $[f]_\bZ$ its $\SL_2(\bZ)$-equivalence class, and denote by $\W_2^\ast(\bZ)$ the set of $\SL_2(\bZ)$-equivalence classes of binary quadratic forms. Put
\[\W_2^\ast(D) = \{w \in \W_2^\ast(\bZ) : \Delta(w) = D\}\]
and put $\O_D$ for the unique quadratic order of discriminant $D$.  It is well-known that the set of primitive classes $[f]_\bZ$ with discriminant $D$ parametrize the ideal classes in the Picard group $\Pic(\O_D)$, the group of ideal classes in $\O_D$ (see \cite{HCL} for a modern treatment). We put 
\begin{equation} \label{class number} h_2(D) = \# \Pic(\O_D). \end{equation} 
We then have the following famous theorem, originally conjectued by Gauss in \cite{Gau} and subsequently proved by Mertens and Siegel \cite{Sieg}: 
\begin{proposition}[Gauss, Mertens, Siegel] \label{GauSieg}The class number $h_2(-D)$ satisfies the following asymptotic formulas:
\begin{equation} \label{full class}  \sum_{0 < D \leq X} h_2(-D) = \frac{\pi}{18 \zeta(3)} X^{3/2} + O(X \log X) \end{equation}
and
\begin{equation} \label{sieg} \sum_{\substack{0 < D \leq X \\ D \equiv 0 \pmod{4} }} h_2(-D) = \frac{\pi}{42 \zeta(3)} X^{3/2} + O(X \log X). \end{equation}
For $D > 0$, we put $R_D = \log \ep_D$ for the \emph{regulator} of the quadratic field $\bQ(\sqrt{D})$. We then have
\begin{equation} \label{full real class} \sum_{0 < D \leq X} h_2(D) R_D = \frac{\pi^2}{18 \zeta(3)} X^{3/2} + O\left(X \log X \right)
\end{equation}
and
\begin{equation} \label{real sieg} \sum_{\substack{0 < D \leq X \\ D \equiv 0 \pmod{4}}} h_2(D) R_D = \frac{\pi^2}{42 \zeta(3)} X^{3/2} + O \left(X \log X \right).
\end{equation}
\end{proposition} 
We remark that our class number (\ref{class number}) only counts \emph{primitive} classes. \\

Recall that a positive definite binary quadratic form $f(x,y) = \alpha x^2 + \beta xy + \gamma y^2$ is said to be \emph{reduced} if its coefficients satisfy $|\beta| \leq \alpha \leq \gamma$. Gauss proved that $h_2(-D)$ is exactly equal to the number of primitive reduced forms of discriminant $-D$. \\

We now require the following lemma, which is useful when estimating the sum of the error terms as we sum across $\SL_2(\bZ)$-classes over positive definite binary quadratic forms:

\begin{lemma} \label{Dav bound} Let $g(x,y) = g_2 x^2 + g_1 xy + g_0 y^2$ be a positive definite and reduced binary quadratic form with co-prime integer coefficients. Let $\sum_{D \leq Y}^\dagger$ denote the sum over positive definite, reduced, and primitive binary quadratic forms $g$ of discriminant up to $Y$. Then
\begin{equation} \label{dagger sum} \sideset{}{^\dagger} \sum_{D \leq X^{2/3}} \frac{1}{(g_2 D)^{1/2}} = O \left(X^{1/2} \right).\end{equation}
\end{lemma}
\begin{proof} We note that the sum (\ref{dagger sum}) is approximated by the integral
\[\mathfrak{I}(X) = \int_1^{X^{1/3}} \int_{a}^{X^{2/3}/a} \int_{-a}^a \frac{db dc da}{a \sqrt{c}}. \]
We evalute $\mathfrak{I}(X)$ by
\begin{align*} \mathfrak{I}(X) & = \int_1^{X^{1/3}} \int_a^{X^{2/3}/a} \frac{2 dc da}{\sqrt{c}} \\
& = \int_1^{X^{1/3}} 4\left(\frac{X^{1/3}}{\sqrt{a}} - a^{1/2} \right) da \\
& = O \left(X^{1/3} \cdot X^{1/6} \right) = O \left(X^{1/2}\right),
\end{align*}
as desired. \end{proof}

Next we state a similar result to Proposition \ref{GauSieg} for $\SL_2(\bZ)$-equivalence classes of reducible forms. 

\begin{proposition} \label{red class} Let $n$ be a positive integer. The number of $\SL_2(\bZ)$-equivalence classes of primitive, integral binary quadratic forms of discriminant $n^2$ is equal to $\phi(n)$, and an explicit set of representatives is 
\[\{ax^2 + n xy : 1 \leq a \leq n-1, \gcd(a,n) = 1\}.\]
Therefore,
\begin{equation} \label{square class} \sum_{n \leq X^{1/2}} h_2(n^2) = \sum_{n \leq X^{1/2}} \phi(n) = \frac{3X}{\pi^2} + O\left(X^{1/2} \log X\right).
\end{equation}
\end{proposition} 

Finally, we have a similar proposition for forms $\GL_2(\bQ)$-equivalent to the form $x^2 + y^2$:

\begin{proposition} Let $n$ be a positive integer. The number of $\SL_2(\bZ)$-equivalence classes fo primitive, integral binary quadratic forms of discriminant $-4n^2$ is given by
\[h_2(-4n^2) = n \prod_{p \equiv 1, 2 \pmod{4}} \left(1 - \frac{1}{p}\right) \prod_{p \equiv 3 \pmod{4}} \left(1 + \frac{1}{p}\right).\]
We further have the asymptotic formula
\begin{equation} \label{neg square class} \sum_{n \leq X^{1/2}/2} h_2(-4n^2) = \frac{3}{32} \left(\sum_{n=0}^\infty \frac{(-1)^n}{(2n+1)^2}\right)^{-1} X + O\left(X^{1/2} \log X\right).\end{equation}
\end{proposition} 

\begin{proof} The first statement is found in \cite{Bue}, pages 109-119. We now prove the asymptotic formula. \\

We denote by $\rho(n) = h_2(-4n^2)$. Then $\rho(n)$ is a multiplicative function and the Dirichlet series of $\rho(n)$ converges absolutely and is holomorphic for $s > 2$. It admits a factorization into the Euler product 
\begin{equation} \label{EP1} \zeta(s-1) \prod_{p \equiv 1,2 \pmod{4}} (1 - p^{-s}) \prod_{p \equiv 3 \pmod{4}} (1 + p^{-s}).  \end{equation}
This follows from the identities
\[\sum_{n=1}^\infty \frac{\phi(n)}{n^s} = \frac{\zeta(s-1)}{\zeta(s)} \text{ and } \sum_{n=1}^\infty \frac{\psi(n)}{n^s} = \frac{\zeta(s)\zeta(s-1)}{\zeta(2s)},\]
where $\psi(n) = n \prod_{p | n} ( 1 + 1/p).$ Let $\chi$ be the unique non-principal character of modulus $4$. Then (\ref{EP1}) can be written as 
\begin{equation} \label{neg square fact} \zeta(s-1) L(s,\chi)^{-1} (1 - 2^{-s}) = \zeta(s-1)\beta(s)^{-1} (1 - 2^{-s}),\end{equation}
where $\beta(s)$ is Dirichlet's beta series. The asymptotic formula (\ref{neg square class}) then follows from Perron's formula. \end{proof} 

We will use the results in this section to allow us to sum over different error terms that arise in the proof of Theorem \ref{MT}.

\section{Some arithmetic and algebraic properties of binary quadratic forms}
\label{quadratic lift}

We will be counting with respect to the $I$-invariant. We now give the $I$-invariant of $F$ for $F$ given as in Proposition \ref{para 1}. By (\ref{disc F}), for all $F \in \V_4(\bR)$ we have
\[\Delta(F) = \frac{4I(F)^3}{27},\]
since $J(F) = 0$. Therefore the condition $|\Delta(F)| \leq X$ is translated into
\begin{equation} \label{I bound} |I(F)|^3 \leq \frac{27 X}{4}.
\end{equation} 
We note that for $\Delta(F) > 0$ for all $F \in \V_4(\bR)$, since the real quadratic form divisor of the Hessian is positive definite. We may thus drop the absolute value in (\ref{I bound}). We have the following lemma:

\begin{lemma} \label{explicit I} Let $f(x,y) = \alpha x^2 + \beta xy + \gamma y^2$ be a binary quadratic form with $\alpha \Delta(f) \ne 0$. Then for $F$ given as in Proposition \ref{para 1}, we have
\begin{equation} \label{I in fam} I(F) = \frac{-3 (\alpha B^2 - 4 \beta AB + 16 \gamma A^2)\Delta(f)}{4 \alpha^3}.\end{equation}
\end{lemma}

We now take
\begin{equation} \label{cal I} \I(F) = \frac{\alpha B^2 - 4 \beta AB + 16 \gamma A^2}{4\alpha^3}.\end{equation}
Now put
\begin{equation} \label{NfX} N_f(X) = \# \{F \in \V_f(\bZ) : \I(F) \leq X\}.
\end{equation}
Since $N_f(X) = N_g(X)$ whenever $f,g$ are $\GL_2(\bZ)$-equivalent, we shall assume from now on that $f$ takes on a convenient form. We shall now take a $\GL_2(\bZ)$-translate of $f$ to be 
\begin{equation} \label{fp}  px^2 + mxy + ny^2,\end{equation}
Here $p$ is the smallest odd prime representable by $f$ not dividing $D$ and $n$ is the smallest positive integer for which (\ref{fp}) holds, and then we choose $m$ to be non-negative. This translate of $f$ is then determined given $p$. \\

We now show, at least when $f$ is given in (\ref{fp}), that whenever $F \in \V_{f}(\bZ)$ that $\I(F) \in \bZ$. We require the following result, which is a special case of Theorem 2 in \cite{Stew}:

 \begin{lemma} \label{only 2 lat} Let $f = \alpha x^2 + \beta xy + \gamma y^2$ be a primitive binary quadratic form with integer coefficients and non-zero discriminant $d$, and let $p$ be an odd prime which is representable by a quadratic form of discriminant $d$ which does not divide $d$. Then there exist linear forms $L_1, L_2$ with coefficients in the $p$-adic integers $\bZ_p$ such that 
 \[f(x,y) = L_1(x,y) L_2(x,y)\]
 over $\bZ_p$. Further, for any positive integer $k$, the solutions to the congruence $f(x,y) \equiv 0 \pmod{p^k}$ with $x,y$ not both zero modulo $p$ lie in exactly one of the two lattices
 \[\L_1^k = \{(x,y) \in \bZ^2 : L_1(x,y) \equiv 0 \pmod{p^k}\}\]
 and
 \[\L_2^k = \{(x,y) \in \bZ^2 : L_2(x,y) \equiv 0 \pmod{p^k}\}.\]. 
\end{lemma}

We call the lattice $\L_i^k$ the $k$-th \emph{Hensel lift} of the lattice $\L_i$. We now prove that the $I(F)/\Delta(f) = \I(F)$ is always an integer whenever $F \in \V_f(\bZ)$.

\begin{lemma} Let $f$ be given as in (\ref{fp}), and let $\I(F)$ be given as in (\ref{cal I}). Then $\I(F) \in \bZ$ whenever $F \in \V_f(\bZ)$. 
\end{lemma}

\begin{proof} With $f$ in the form given in (\ref{fp}), the congruence condition (\ref{cong 3}) implies (\ref{cong 2}) and (\ref{cong 1}). The only two primes that need to be considered are 2 and $p$ itself. Note that since we assumed $p \nmid \Delta(f)$, it follows that $p \nmid m$. Therefore examining $\A_3 \equiv 0 \pmod{p}$ yields
\begin{align*} 4n(m^2 - pn)A - m(m^2 - 2pn)B & \equiv  m^2(4nA - mB) \\
& \equiv 0  \pmod{p},
\end{align*}
hence
\[4nA - mB \equiv 0 \pmod{p}\]
or (\ref{cong 1}). Now reducing modulo $p^2$ we get
\[(4m^2nA - m^3 B + mnp B) - pn(4n A - mB) \equiv 0 \pmod{p^2},\]
and the second term vanishes mod $p^2$, hence the first term must vanish as well. However the first term is equal to $m(4mn A - (m^2 - pn)B)$, which implies (\ref{cong 2}). Moreover, we see that (\ref{cong 3}) is soluble with $A \not \equiv 0 \pmod{p}$, whence $\L_{f,p}$ is contained in the 3rd Hensel lift of the lattice defined by (\ref{cong 1}), whence $p^3 | pB^2 - 4 m AB + 16 n B^2$ whenever $(A,B) \in \L_{f,p}$. \\ 

We now have to deal with the prime 2. If $m$ is odd, then $\A_3 \equiv 0 \pmod{4}$ is equivalent to $B \equiv 0 \pmod{4}$ and $\A_1 \equiv 0 \pmod{2}$ is equivalent to $B \equiv 0 \pmod{2}$. We then see at once this is sufficient for the numerator of $\I(F)$ to be divisible by 4. This proves the claim. 
\end{proof}

We can now clear the denominator in (\ref{cal I}) by restricting to the lattice $\L_{f,p}$ to get a new quadratic form $\nu(f)$. A potential problem is that we do not know about the $\SL_2(\bZ)$-equivalence class of $\nu(f)$ given $f$. It turns out that this would create problems when estimating the contribution of $N_f(X)$ with $D = -\Delta(f)$ large. Therefore, we must resolve this issue. 

\subsection{Auxiliary quadratic forms $w(f)$ and $\nu(f)$} We now define certain auxiliary quadratic forms $w(f), \nu(f)$ for quadratic forms given in (\ref{fp}). Suppose that $m$ is odd. Then for a quartic form $F$ given by (\ref{J family}) to have integer coefficients, we must have $B \equiv 0 \pmod{4}$. No congruence conditions modulo a power of 2 is imposed if $m$ is even. Thus, when $m$ is odd we shall assume $B \equiv 0 \pmod{4}$, which changes (\ref{I in fam}) to
\begin{equation} \label{I in fam even} \frac{4(p B^2 - m AB +  n A^2)\Delta(f)}{p^3}.\end{equation}
We then put
\begin{equation} \label{numerator} w(f)(x,y) = \begin{cases} px^2 - mxy + ny^2 & \text{if } m \text{ is odd} \\ px^2 -  4m xy + 16ny^2 & \text{if } m \text{ is even}. \end{cases} 
\end{equation}
When $m$ is odd, put
\begin{equation} \label{L2 odd} \L_2(w(f)) = \{(x,y) \in \bZ^2 : mx \equiv ny \pmod{p}\},\end{equation}
and $\L_2^3(w(f))$ for its 3rd Hensel lift. Likewise, when $m$ is even, put
\begin{equation} \label{L2 even} \L_2(w(f)) = \{(x,y) \in \bZ^2 : mx \equiv 4ny \pmod{p}\} \end{equation}
and $\L_2^3(w(f))$ for its 3rd Hensel lift. If $\displaystyle \left\{\binom{u_1}{v_1}, \binom{u_2}{v_2} \right\}$ is a basis for $\L_2^3(w(f))$, then the matrix $U = \left(\begin{smallmatrix} u_1 & u_2 \\ v_1 & v_2 \end{smallmatrix} \right)$ satisfies
\begin{equation} \label{key transform} w(f)(u_1 x + u_2 y, v_1 x + v_2 y) = p^3 g(x,y) \end{equation}
for some primitive binary quadratic form $g$. We then have the following equality:
\begin{equation} \label{nuf} \{w(f)(x,y) : (x,y) \in \L_2^3(f)\} = \{p^3 g(u,v) : (u,v) \in \bZ^2\}.\end{equation}
Moreover, $g$ is well-defined up to $\GL_2(\bZ)$-equivalence and $\Delta(w(f)) = \Delta(g)$ by (51) in \cite{Stew}. We shall denote this $g$ (or any $\GL_2(\bZ)$-translate of it) by $\nu(f)$. \\

We may now state the main proposition of this subsection:

\begin{proposition} \label{distinct prop} Let $f$ be a positive definite binary quadratic form with co-prime integer coefficients and discriminant $-D$, and let $w(f)$ be given as in (\ref{numerator}). Then there exists a quadratic form $\nu(f)$ which satisfies (\ref{nuf}) with $\Delta(w(f)) = \Delta(\nu(f))$ and $[w(f)]_\bZ^4 = [\nu(f)]_\bZ$. Moreover, for primitive $g \in V_2(\bZ)$ with discriminant $\Delta(f)$, we have $w(f), w(g)$ are $\GL_2(\bZ)$-equivalent if and only if $f,g$ are $\GL_2(\bZ)$-equivalent. 
\end{proposition}

The proof of Proposition \ref{distinct prop} relies on a refinement of Hensel's lemma applied to prime ideals. We first recall the following classical fact: 

\begin{lemma} \label{rep cond} Let $f$ be a positive definite binary quadratic form with co-prime integer coefficients and discriminant $-D$, and let $m$ be a positive integer. Then $m$ can be represented by $f$ if and only if the principal ideal $(m)$ admits a factorization as $\fm \ol{\fm}$, where $\fm$ is an ideal in $\O_{-D}$ in the ideal class parametrized by $[f]_\bZ$. 
\end{lemma}
Note that for a prime $p$, the principal ideal $(p)$ is either inert in $\O_{-D}$ or splits into $\fp_1 \fp_2$. Then $p$ is representable by $f$ if and only if it splits in $\O_{-D}$ and either $\fp_1$ or $\fp_2$ lies in the ideal class parametrized by $[f]_\bZ$. Note that the ideal class of $\fp_1$ is the inverse of $\fp_2$ in the ideal class group, since their product is a principal ideal. We thus have the following corollary to Lemma \ref{rep cond}: 

\begin{corollary} Le $D$ be a positive integer and let $p$ be a prime not dividing $D$. Then $p$ is represented by exactly two $\SL_2(\bZ)$-classes of primitive quadratic forms of discriminant $-D$ up to multiplicity, and these ideal classes are inverses of each other in the ideal class group. 
\end{corollary}
Now note the following simple observation: if $f(x,y) = ax^2 + bxy + cy^2$, then a representative of its inverse class $[f]_\bZ^{-1}$ is given by $ax^2 - bxy + cy^2$. In particular, the classes $[f]_\bZ$ and $[f]_\bZ^{-1}$ are $\GL_2(\bZ)$-equivalent. \\



We now show that the forms $w(f)$ given in (\ref{numerator}) are distinct, which proves one part of Proposition \ref{distinct prop}.

\begin{lemma} \label{wf distinct} Let $f,g$ be two binary quadratic forms with co-prime integer coefficients and equal discriminant. Then the forms $w(f), w(g)$ given in (\ref{numerator}) are $\GL_2(\bZ)$-distinct if and only if $f$ and $g$ are $\GL_2(\bZ)$-distinct.
\end{lemma}

\begin{proof} It is clear that if $f$ and $g$ are $\GL_2(\bZ)$-equivalent, then so are $w(f)$ and $w(g)$. For the converse, first suppose that $m$ is odd. Then $w(f)$ is $\GL_2(\bZ)$-equivalent to $f$ via $\left(\begin{smallmatrix} 1 & 0 \\ 0 & -1 \end{smallmatrix} \right)$, hence the statement is clear. Now suppose that $m$ is even. Put $g = g_2 x^2 + g_1 xy + g_0 y^2$, where $g_2$ is an odd prime. Note that $g_1$ is even, since $\Delta(f) = \Delta(g)$ and the parity only depends on the middle coefficient. It then follows that
\[w(g) = g_2 x^2 - 4 g_1 xy + 16 g_0 y^2.\]
We suppose that $w(f)$ is equivalent to $w(g)$, and let $T = \left(\begin{smallmatrix} t_1 & t_2 \\ t_3 & t_4 \end{smallmatrix} \right) \in \GL_2(\bZ)$ be such that
\[w(f)_T(x,y) = p (t_1 x + t_2 y)^2 - 4 m (t_1 x + t_2 y)(t_3 x + t_4 y) + 16n (t_3 x + t_4 y)^2 = w(g)(x,y).\]
Since $g_2$ is odd, it follows that $t_1$ is odd. Moreover, the middle coefficient of $w(f)_T$ is equal to
\[ 2 (pt_1 t_2 - 2 m t_2 t_3 - 2 m t_1 t_4 + 16 n t_3 t_4).\]
We need this to be divisible by 8, since $8 | 4 g_1$. This implies that $2p t_1 t_2 \equiv 0 \pmod{8}$. However $p, t_1$ are odd, so $t_2 \equiv 0 \pmod{4}$. It then follows that 
\[T' = \begin{pmatrix} 1 & 0 \\ 0 & 4 \end{pmatrix} T \begin{pmatrix} 1 & 0 \\ 0 & 1/4 \end{pmatrix} = \begin{pmatrix} t_1 & t_2/4 \\ 4 t_3 & t_4 \end{pmatrix} \in \GL_2(\bZ),\]
which shows that $f$ and $g$ are equivalent.   
 \end{proof}

\subsection{Hensel lifting of lattices and $\SL_2(\bZ)$-classes of binary quadratic forms} 

We shall treat forms $f$ given in the shape (\ref{fp}) (note that $w(f)$ is also of the shape (\ref{fp})). Put

\begin{equation} \label{L1} \Lambda_1(f) = \{(x,y) \in \bZ^2 : y \equiv 0 \pmod{p}\} \end{equation}
and
\begin{equation} \label{L2} \Lambda_2(f) = \{(x,y) \in \bZ^2 : mx + ny \equiv 0 \pmod{p}\}.\end{equation}
Following the notation of Lemma \ref{only 2 lat}, we put $\Lambda_i^k(f)$ for the $k$-th Hensel lift of the lattice $\Lambda_i(f)$. \\
 
For each $i$ and $k$, we assign a class $w \in W_2(\bZ)$ to $\Lambda_i^k$ as follows. There exists a quadratic form $g_{i,k}$, unique up to $\GL_2(\bZ)$-equivalence, such that
\begin{equation} \label{Hensel g} \{f(x,y) : (x,y) \in \Lambda_i^k\} = \{p^k g_{i,k}(x,y) : (x,y) \in \bZ^2\}.\end{equation}
The form $g_{i,k}$ has the same discriminant as $f$ and hence $[g_{i,k}]_\bZ$ is in the same ideal class group as $[f]_\bZ$. 

\begin{lemma} \label{class lemma} Let $f$ be given as in (\ref{fp}). Then for each $k \geq 1$ one can choose an integral binary quadratic form $g_{i,k}$ satisfying (\ref{Hensel g}) such that $[g_{1,k}]_\bZ = [f]_\bZ^{k-1}$ and $[g_{2,k}]_\bZ = [f]_\bZ^{k+1}$. 
\end{lemma}

We shall state a slightly more general result, which may be of separate interest. From here on, $f$ shall not be assumed to take the form (\ref{fp}). We hence return to the notation (\ref{f-alpha}) for $f$. 

\begin{proposition} \label{hensel lift prop} Let $f$ be a binary quadratic form with co-prime integer coefficients and non-zero discriminant. Suppose that $p$ is an odd prime such that $\left(\frac{\Delta(f)}{p}\right) = 1$, and let $\fp_1, \fp_2$ be the two prime ideal divisors of $(p)$. Suppose further that there exists a non-negative integer $s$ for which $[f]_\bZ = [\fp_1^s]$. Let $\L_1, \L_2$ be as given in Lemma \ref{only 2 lat}. For all $k \geq 1$, there exists integral quadratic forms $g_{1,k}, g_{2,k}$ such that
\[\{f(x,y) : (x,y) \in \L_i^k\} = \{p^k g_{i,k}(x,y) : (x,y) \in \bZ^2\}\]
for $ i = 1,2$ and 
\[[g_{1,k}]_\bZ = [\fp_1]^{s - k}, [g_{2,k}]_\bZ = [\fp_1]^{s+k}.\]
\end{proposition}

\begin{proof} See the Appendix by Erick Knight. 
\end{proof}

We then have the following corollary, which a crucial component of our proof:

\begin{corollary} \label{4power} Let $f$ be given as in (\ref{fp}), $w(f)$ as given in (\ref{numerator}). Then $\nu(f)$ can be chosen so that $[\nu(f)]_\bZ = [w(f)]_\bZ^4$.  
\end{corollary}

The proof of Proposition \ref{distinct prop} then follows from Corollary \ref{4power} and Lemma \ref{wf distinct}.

\section{Enumerating the elements in $\V_f(\bZ)$ for a fixed positive definite quadratic form $f$}
\label{total real}

In this section we shall give an asymptotic formula for $N_f(X)$. 
\begin{theorem} \label{pos def count} Let $f(x,y) = \alpha x^2 + \beta xy + \gamma y^2$ be a positive definite reduced binary quadratic form with co-prime integer coefficients. Put $D = |\Delta(f)|$. Then 
\begin{equation} \label{pen} N_f(X/D) =  \begin{cases} \dfrac{\pi X}{3 D^{3/2}}  + O \left(\dfrac{X^{1/2}}{D^{1/2}}\right) & \text{if } \beta \text{ is odd}; \\ \\
\dfrac{4\pi X}{3 D^{3/2}} + O \left(\dfrac{X^{1/2}}{D^{1/2}} \right) & \text{if } \beta \text{ is even}. \end{cases}\end{equation}
\end{theorem} 

The following lemma illustrates the importance of having $f$ as a positive definite form. 

\begin{lemma} \label{ellipse eq} Let $f$ be a primitive, positive definite binary quadratic form with integer coefficients. Then the $I$-invariant of $F$ is a positive definite binary quadratic form and hence the set of $F \in \V_f(\bR)$ with $I(F) \leq X^{1/3}$ lie in the ellipse defined by
\begin{equation}  \E_f(X) = \left \{(A,B) \in \bR^2 : \alpha B^2 - 4 \beta AB + 16 \gamma A^2 \leq \frac{4 \alpha^3}{3 |\Delta(f)|} X^{1/3} \right\}.\end{equation}
\end{lemma} 

\begin{proof} The fact that the $I$-invariant is a positive definite binary quadratic form follows from (\ref{I in fam}), namely the observation that the $I$-invariant is a positive multiple of the quadratic form $\alpha x^2 - 4 \beta xy + 16 \gamma y^2$, which is equal to $f(x, -4y)$; hence positive definite. 
\end{proof}

Since $N_f(X)$ only depends on the class $\C(f)$ of $f$, we may pick a convenient representative of $f$. Indeed, we may suppose that $\alpha$ is odd and co-prime to $\Delta(f)$, which implies that $\gcd(\alpha, \beta) = 1$. \\ \\
It follows from Lemma \ref{ellipse eq} and Proposition \ref{para 1} that 
\[\{F \in \V_{f,\bZ}: I(F) \leq X\} = \E_f \cap \L_{f,\alpha}.\]
We then need to compute the determinant $\det (\L_{f,\alpha})$, which we do so in the following proposition. 

\begin{proposition} \label{det calc} The determinant $\det \L_{f,\alpha}$ is equal to $4\alpha^3$ if $\beta$ is odd and $\alpha^3$ otherwise. 
\end{proposition} 

\begin{proof} Let $k$ be the exponent of $p$ dividing $\alpha$. For $p \geq 3$, the congruence (\ref{cong 3}) implies
\[\beta(4 \beta \gamma A- (\beta^2 - \alpha \gamma)  B) - \alpha \gamma(4 \gamma A - \beta B) \equiv 0 \pmod{p^{3k}}.\]
It follows that
\[4 \beta \gamma A - (\beta^2 - \alpha \gamma) B \equiv 0 \pmod{p^k},\]
which is equivalent to
\[4 \gamma A - \beta B \equiv 0 \pmod{p^k},\]
or $\A_1 \equiv 0 \pmod{p^k}$. This then implies that
\[\beta (4 \beta \gamma A - (\beta^2 - \alpha \gamma) B) \equiv 0 \pmod{p^{2k}},\]
or $\A_2 \equiv 0 \pmod{p^{2k}}$. Thus (\ref{cong 3}) implies (\ref{cong 2}) and (\ref{cong 1}), so that 
\[\det \L_{f,\alpha}^{(p)} = p^{3k}.\]
We now deal with the case when $p = 2$. If $2 \nmid \beta$ then the argument proceeds as before, but  the congruences modulo $p^k, p^{2k}, p^{3k}$ are replaced by $2^{k+1}, 2^{2k}, 2^{3k+2}$ respectively. 
\end{proof} 
We may now prove Theorem \ref{pos def count}. 
\begin{proof}[Proof of Theorem \ref{pos def count}] We wish to count the number of points in the intersection $\E_f(X) \cap \L_{f,\alpha}$. We apply Davenport's lemma, which asserts that
\[\# \E_f(X) \cap \L_{f,\alpha} = \frac{\Vol(\E_f(X))}{|\det(\L_{f,\alpha})|} + O \left(\max\left\{\Vol(\ol{\E_f(X)}), 1 \right \} \right).\]
The volume of $\E_f(X)$ is given by
\[\Vol(\E_f(X)) = \frac{4 \alpha^3}{3 |\Delta(f)|^{3/2}} X^{1/3},\]
whence
\[N_f^{(0)}(X) = \begin{cases} \dfrac{\pi X^{1/3}}{3 |\Delta(f)|^{3/2}}  + O \left(\dfrac{X^{1/6}}{|\Delta(f)|^{1/2}}\right) & \text{if } \beta \text{ is odd}; \\ \\
\dfrac{4\pi X^{1/3}}{3|\Delta(f)|^{3/2}} + O \left(\dfrac{X^{1/6}}{ |\Delta(f)|^{1/2}} \right) & \text{if } \beta \text{ is even}, \end{cases}\]
since $\alpha \leq \gamma$ since $f$ is reduced. 

 \end{proof}

The task now, given Theorem \ref{pos def count}, is to obtain uniformity estimates for the error term that appears. We will need the following lemma, essentially Lemma 3.1 in \cite{BG}. 

\begin{lemma} \label{BG lem} Let $f(x,y) = \alpha x^2 + \beta xy + \gamma y^2$ be a positive definite integral binary quadratic form which is reduced. Then the set of integer pairs $(x,y)$ satisfying $f(x,y) \leq X$ is given by 
\[\frac{2 \pi X}{\sqrt{4 \alpha \gamma - \beta^2}} + O \left(\sqrt{\frac{X}{\alpha}} \right).\]
\end{lemma}

\section{Two ways to count quartic forms with vanishing $J$-invariant} 
\label{strat} 

In Section \ref{total real} we showed how to estimate the quantities $N_f(X)$. By summing over $\SL_2(\bZ)$-classes of $f$ having discriminant $-D$ bounded by $X$, we will be able to prove Theorem \ref{MT}. However, this is infeasible: as soon as $D > X^{2/3}$ the main term in (\ref{pen}) will be less than one. It then becomes un clear how continuing to sum $N_f(X)$ past that point will contribute to a main term. \\

We put $N(X; Y)$ for the quantity (\ref{NX}), with the additional stipulation that we are only counting contributions from those $N_f(X)$ with $|\Delta(f)| \leq Y$. Using the theory introduced in Section \ref{outer para}, we will introduce a different way to enumerate elements in $\W_4^{(0)}(\bZ)(X)$, via the \emph{outer form} $h(x,y)$ in (\ref{outer para eq}). In particular, put
\begin{equation} \S_h(X) = \# \{w \in \W_4^{(0)}(\bZ) : 0 < |\Delta(w)| \leq X, \exists \text{ primitive quadratic forms } u,v \text{ s.t. } h(u(x,y), v(x,y)) \in w \}
\end{equation} 
Recall (\ref{outer L}) the relation
\[\Delta(u) h_2 - \Delta(u,v) h_1 + \Delta(v) h_0 = 0.\] 
Using the fact that $-D = \Delta(\J(u,v))/4$, we see that we can express $-D$ as a function of the quadratic form $h(x,y) = h_2 x^2 + h_1 xy + h_0 y^2$ and $\Delta(u), \Delta(u,v), \Delta(v)$, namely by setting
\[\Delta(v) = - \frac{h_2 \Delta(u) - h_1 \Delta(u,v)}{h_0}.\]
This gives
\begin{equation} \label{-D new} - D = \frac{h_2 \Delta(u)^2 - h_1 \Delta(u) \Delta(u,v) + h_0 \Delta(u,v)^2}{h_0} = \frac{h(\Delta(u), -\Delta(u,v))}{h_0}.
\end{equation}
Thus, we have the key equation
\begin{equation} \label{I h form} I(F) = \frac{-3h(\Delta(u), - \Delta(u,v))}{h_0} \cdot \Delta(h).
\end{equation}
Thus we can formulate the question as finding pairs $(x,y) \in \bZ^2$ for which 
\begin{equation} \label{h cong} h(x,y) \equiv 0 \pmod{h_0} \end{equation}
and such that
\begin{equation} \label{outer ineq} \left \lvert \frac{h(x,y)}{h_0} \right \rvert \leq \frac{X}{\Delta(h)}.\end{equation}
The caveat is that not all such pairs $(x,y)$ are admissible: indeed, we shall only take those pairs $(x,y)$ satisfying the congruence condition (\ref{h cong}) and such that one can find a primitive pair of quadratic forms $(u,v)$ such that $\Delta(u) = x, \Delta(u,v) = -y$. Luckily such a criterion is already known, due to Morales \cite{Mor}: such a pair exists if and only if $x$ is a square modulo $-D$. We state Morales' result for convenience:

\begin{proposition}[Morales \cite{Mor}] \label{Mor prop} The number of inner equivalence classes of pairs $(u,v)$ of integral binary quadratic forms with prescribed invariant form $\F(x,y) = \delta_1 x^2 + 2 \delta_{1,2} xy + \delta_2 y^2$ is equal to 
\begin{equation} \sum_{\substack{c | \delta_{1,2}^2 - \delta_1 \delta_2 \\ c > 0, c \text{ square-free}}} \left(\frac{\delta_1}{c}\right). 
\end{equation} 
\end{proposition}

Note that since $h$ is indefinite, the number of solutions to (\ref{outer ineq}) is a priori infinite. Thus to make sense of the counting problem we must account for the action of the unit group of $\O_h = \O_{\mathbb{Q}(\sqrt{\Delta(h)})}$ on $h$, so that at most a bounded number of points from each orbit is counted. To do so we must define such an action, but unfortunately both the inner and outer actions introduced in Section \ref{outer para} do not have an immediate interpretation in terms of the expression introduced in (\ref{-D new}). Fortunately, with respect to the unit group action, the outer action induces the correct action on (\ref{-D new}). Therefore, Proposition \ref{Mor prop} implies the following:
\begin{proposition} Let $h(x,y)$ be a primitive integral binary quadratic form such that $\Delta(h) > 0$. Then
\begin{equation} \label{h sum} \S_h(X) = \sideset{}{^\ast} \sum_{\substack{-Xh_0 < h(x,y) < 0 \\ h_2 x - h_1 y \equiv 0 \pmod{h_0} }} \sum_{\substack{c | h(x,y) \\ c > 0, c \text{ square-free}}} \left(\frac{x}{c}\right),
\end{equation} 
where the summation $\sideset{}{^\ast} \sum$ denotes summing over a suitable fundamental domain $\fD$ for the action of the unit group in $\bQ(\sqrt{\Delta(h)})$ on $h$. 
\end{proposition}
What will be important for us is the following. Put $\N(X;Y)$ for the sum of $\S_h(X)$ over $\SL_2(\bZ)$-equivalence classes of $h$ with $E = \Delta(h) \leq Y$. We then have:

\begin{proposition} \label{two sums} For any positive number $Y$ we have
\[N(X) = N\left(X; Y \right) + \N\left(X; XY^{-1} \right).\]
\end{proposition}

\begin{proof} If $w \in \W_4^{(0)}(\bZ)$ is such that $w$ is counted by both $N_f(X)$ and $\S_h(X)$, then
\[I(F) = - 3 \Delta(f) \Delta(h).\]
Thus, if $|I(F)| \leq X$, then $F$ is counted by either $N_f(X)$ with $|\Delta(f)| \leq Y$ or $\S_h(X)$ with $|\Delta(h)| \leq XY^{-1}$. This completes the proof. 
\end{proof} 

Therefore to complete the proof of Theorem \ref{MT}, we need to estimate $\N(X; X^{1/3})$. To do so, we need to estimate $\S_h(X)$ with reasonable precision. \\

Indeed a suitable choice of fundamental domain is crucial for the necessary estimates: for this we follow a construction due to Schmidt \cite{Schi}. Recall that for $E = \Delta(h)$, we denote by $R_E$ the regulator of the quadratic order $\O_E$ of discriminant $E$. Following \cite{Schi}, we put
\[t_E = \lfloor R_E \rfloor + 1, u = \exp(R_E/t_E).\]
Then we have
\[u^{t_E} = \ep_{E} \text{ and } 1 \ll \log u \leq 1,\]
where $\ep_E$ is the fundamental unit of $\O_E$. Instead of considering sublattices of $\bZ^2$, we instead consider, for points $\alpha \in K_E = \bQ(\sqrt{E})$, the put
\[\widehat{\alpha} = \left(\alpha, \ol{\alpha}\right) \in \bR^2,\]
where $\ol{\alpha}$ is the conjugate of $\alpha$ in $\O_E$. Note that the set of $\widehat{\alpha}, \alpha \in \O_E$ gives a sub-lattice of $\bR^2$ of discriminant $E^{1/2}$. When $\alpha$ is restricted to a non-zero ideal $\fa \subset \O_E$, then $\widehat{\alpha}$ runs over a lattice $\Lambda(\fa)$ satisfying $\det \Lambda(\fa) = E^{1/2} \fN(\fa)$, where $\fN(\fa)$ refers to the norm of the ideal $\fa$. Further, for $\fC$ an ideal class in $\O_E$ and $\fc_1, \fc_2, \cdots$ the integral ideals in $\fC$ ordered by norm, put
\[\fN(\fC) = \left(\sum_{j=1}^{2t} \fN(\fc_j) \right)^{-1/2} .\]
Put $v = u^{1/2}$, so that $1 \ll \log v$, and $v - 1 \gg 1$. Put 
\[\tau : \bR^2 \rightarrow \bR^2, \tau(\alpha, \ol{\alpha}) = \left(v^{-1} \alpha, v \ol{\alpha}\right).\]
Put $\Lambda(\fa, j)$ for the image of $\Lambda(\fa)$ under the map $\tau^j$, where the exponent refers to functional composition. Then $\Lambda(\fa, j)$ is again a lattice in $\bR^2$, since $\tau$ is a linear map. Moreover we have $\det \Lambda(\fa, j) = \det \Lambda(\fa)$. \\

What we gain is that Schmidt shows in \cite{Schi} that we have a very nice expression for the first successive minimum of $\Lambda(\fa, j)$, given by
\begin{equation} \label{first min} \lambda_1 \left(\fa, j \right) = \min_{\alpha \in \fa \setminus \{0\}} \left(v^{-2j} \lvert \alpha \rvert^2 + v^{2j} \lvert \ol{\alpha} \rvert^2 \right)^{1/2}.
\end{equation}

Now put, for $\alpha \in K_E$, 
\[\psi(\alpha) = \frac{\lvert \alpha \rvert}{\lvert \ol{\alpha} \rvert}.\]
By explicit calculation we see that $\psi\left(\ep_E \cdot \alpha\right) = \ep_E^2 \psi(\alpha)$ for all $\alpha \in \O_E$. This shows that for each such $\alpha$ there exists uniquely an integer $s$ such that
\[\ep_E^{-1} < \psi(\ep_E^s \alpha) \leq \ep_E.\]
A key observation made by Schmidt is that the interval 
\[\ep_E^{-1} < x \leq \ep_E\]
may be partitioned into $2t$ intervals $u^{j-1} < x \leq u^j$ with $-t < j \leq t$. Using Schmidt's notation, put $Z_1(\fa, j, X)$ for the number of non-zero $\alpha \in \O_E \cap \fa$ satisfying $|\alpha \ol{\alpha}| \leq X \fN(\fa)$ and $u^{j-1} < \psi(\alpha) \leq u^j$. Lemma 6 in \cite{Schi} gives the estimate 
\begin{equation} \label{Schi lem 6} Z_1(\fa, j, X) = \frac{2 R_E X}{t_E \cdot E^{1/2}} + O \left(\frac{X^{1/2} \fN(\fa)^{1/2}}{\lambda_1(\fa, j)}\right).
\end{equation}
Schmidt's Lemma 8 provides a key estimate, namely
\begin{equation} \sum_{j = 1 - t}^t \lambda_1(\fa, j)^{-1} = \left(\fN(\fC_\fa) \fN(\fa) \right)^{-1/2},
\end{equation}
where $\fC_\fa$ is the ideal class containing $\fa$. \\

We note that for any positive integer $k$ and integer $x$ that
\[\sum_{\substack{c | k \\ c \leq \sqrt{k}}}  \left(\frac{x}{c}\right) = \sum_{\substack{c | k \\ c \geq \sqrt{k}}} \left(\frac{x}{c}\right).\]
Using this observation, we may essentially apply Dirichlet's hyperbola trick to (\ref{Mor prop}) to obtain
\begin{equation} \label{folded sum} \S_h(X)  = \S_h^{(1)}(X) + 2 \sideset{}{^\sharp} \sum_{2 \leq c \leq X^{1/2}} \sideset{}{^\ast} \sum_{\substack{-Xh_0 < h(x,y) < 0 \\ h(x,y) \equiv 0 \pmod{c} \\ h_2 x \equiv h_1 y \pmod{h_0}}} \left(\frac{x}{c}\right),
\end{equation}
where the sum $\sideset{}{^\sharp} \sum$ in (\ref{folded sum}) refers to the restriction to square-free $c$. Denote this sum by $\S_h^\sharp(X)$. \\

To proceed, we note that, by Proposition \ref{hensel lift prop}, we may find a binary quadratic form $\fh$ such that
\begin{equation} \{h(x,y) : (x,y) \in \bZ^2, h_2 x - h_1 y \equiv 0 \pmod{h_0}\} = \{h_0 \cdot \fh(x,y) : (x,y) \in \bZ^2\}.
\end{equation}
After replacing $h$ with $\fh$, the congruence condition modulo $h_0$ then disappears, and the sum $\S_h^\sharp(X)$ becomes
\begin{equation} \label{h conged} \S_h^\sharp(X) = \sideset{}{^\sharp} \sum_{2 \leq c \leq X^{1/2}} \sideset{}{^\ast}\sum_{\substack{-X < \fh(x,y) < 0 \\ \fh(x,y) \equiv 0 \pmod{c}}} \left(\frac{\ell(x,y)}{c}\right)
\end{equation}
for some primitive linear form $\ell$. \\

We break the inner sum of (\ref{h conged}) into a summation over values $b$ of $\ell$, and then summing values of $x,y$ on the line defined by $\ell(x,y) = b$. By the preceding discussion, $x,y$ are constrained in $2t$ domains each with diameter $O\left(X^{1/2} \fN(\fC_\fh)^{-1/2} \right)$ (see Lemma 6 and Lemma 11 in \cite{Schi}) . In particular, we have
\begin{align*} \sideset{}{^\ast}\sum_{\substack{-X < \fh(x,y) < 0 \\ \fh(x,y) \equiv 0 \pmod{c}}} \left(\frac{\ell(x,y)}{c}\right) & = \sum_{b \ll X^{1/2} \fN(\fC_{\fh})^{-1/2} } \left(\frac{b}{c}\right) \sum_{\substack{(x,y) \in \fD(X) \\ \ell(x,y) = b \\ \fh(x,y) \equiv 0 \pmod{c}}} 1 \\
& \ll  \sqrt{c} \log c \cdot \left(\frac{X^{1/2}}{c\fN(\fC_{\fh})^{1/2}} + O(1) \right) 2^{\omega(c)} \\
& \ll_\ep X^{1/2} \fN(\fC_{\fh})^{-1/2} c^{-1/2 + \ep} 
\end{align*} 
Feeding this back into (\ref{folded sum}), and noting that $c$ only runs over norms of $\bQ(\sqrt{\Delta(h)})$, and such integers are bounded by $O((X/\Delta(h))^{1/2})$, we see that 
\begin{align*} \S_h^\sharp(X) & \ll_\ep X^{1/2} \fN(\fC_{\fh})^{-1/2} \sum_{\substack{c \leq X^{1/2} \\ c \text{ a norm in } \bQ(\sqrt{\Delta(h)})}} c^{-1/2 + \ep}  \\
& \ll_\ep X^{1/2} \fN(\fC_{\fh})^{-1/2} \cdot \frac{X^{1/4 + \ep} }{\Delta(h)^{1/2}} \\
& = O_\ep \left(X^{3/4 + \ep}  (\fN(\fC_{\fh}) \Delta(h))^{-1/2}\right).
\end{align*}
The sum corresponding to the value $c = 1$, which we denoted by $\S_h^{(1)}(X)$ in (\ref{folded sum}), is readily seen to be equal to
\[ \S_h^{(1)}(X) = \# \{(x,y) \in \mathfrak{D} : - Xh_0 < h(x,y) < 0, h(x,y) \equiv 0 \pmod{h_0} \}  \]
where $\fD$ denotes a fundamental domain of the action of the unit group on the ring of integers of $\bQ(\sqrt{\Delta(h)})$. By the proof of Theorem 1 in \cite{Schi}, and putting $E = \Delta(h)$, we see that
\begin{equation} \label{Sh1} \S_h^{(1)}(X) = \frac{2X R_E }{\sqrt{E}} + O \left( \left(X (\log X)  R_E \right)^{1/2} \right).
\end{equation}
This leads to the following conclusion:

\begin{lemma} \label{Sh count} Let $h$ be an ntegral binary quadratic form with discriminant $E > 0$, and let $R_E$ be the regulator of $\bQ(\sqrt{E})$. Then 
\[\S_h(X) = \frac{2X R_E}{\sqrt{E}} + O_\ep \left(\left(X(\log X) R_E\right)^{1/2} + X^{3/4+\ep}  (R_E/E)^{1/2} \right) \]
\end{lemma}

\section{Proof of Theorem \ref{MT}}
\label{MT proof} 

By Proposition \ref{two sums}, it suffices to give $N(X; Y)$ and $\N(X ; XY^{-1})$ for any $0 < Y < X$. We shall do so for $Y = X^{2/3} \exp \left(\frac{-4\log X}{\log \log X} \right) $. \\

By Lemma \ref{BG lem}, we see that the error term in $N_f(X/D)$ can be taken to be 
\[O \left(\frac{X^{1/2}}{(\nu_2 D)^{1/2}}\right),\]
where $\nu_2$ is the leading coefficient of $\nu(f)$ is the smallest positive intege representable by $\nu(f)$. Lemma \ref{Dav bound} then shows that
\begin{equation} \label{pos def err} \sideset{}{^\dagger}\sum_{D \leq Y} \frac{X^{1/2}}{(\nu_2 D)^{1/2}} = O\left(X \exp \left(\frac{-2 \log X}{\log \log X}\right) \right).\end{equation}
However, we note that the set of possible classes for $\nu(f)$, by Proposition \ref{hensel lift prop}, is not all possible classes of discriminant $-D$ but rather restricted to those which are $4$-th powers in the class group of forms of discriminant $-D$. For each $4$-th power, we then need to multiply by the size of the $4$-torsion subgroup of the group of forms of discriminant $-D$ to recover all possible classes of $f$. The size of the $4$-torsion subgroup is at most the square of the size of the $2$-torsion subgroup, which by genus theory is at most $2^{\omega(D)}$, where $\omega(\cdot)$ is the number of distinct prime divisor function. Therefore, we need to multiply (\ref{pos def err}) by the largest possible size of the divisor function of a positive integer smaller than $X$, which is of size $O \left(X^{2/\log \log X}\right)$. We chose our $Y$ to balance this contribution, so that the overall contribution is $O(X)$. \\

Next, we need to sum over the complementary error terms coming from Lemma \ref{Sh count}. Examining Schmidt's proof of Lemma 9 in \cite{Schi} reveals that the geometry of the fundamental domain $\fD$ is affected by the norm of the ideal associated to the norm form $h$, and we see that the first error term in Lemma \ref{Sh count} can be handled as
\begin{align*} \sum_{E \leq Z} \frac{(X (\log X) h_2(E) R_E)^{1/2}}{E^{1/2}} & \leq (X \log X)^{1/2} \left(\sum_{E \leq Z} E^{-1} \right)^{1/2} \left(\sum_{E \leq Z} h_2(E) R_E \right)^{1/2} \text{ by Cauchy-Schwarz} \\
& \ll (X \log X)^{1/2} (\log Z)^{1/2} Z^{3/4}. 
\end{align*}
Taking $Z = X^{1/3 + 1/100} > X^{1/3} \exp \left(4 (\log X)(\log \log X)^{-1}\right)$, we see that this gives a negligible contribution. \\

Now we sum the second error term in Lemma \ref{Sh count}. Similarly, summing over the classes of discriminant $E$ has the effect of giving the inclusion of the square-root of the class number $h_2(E)$, giving the error term 
\[O \left( X^{3/4} (\log X) (h_2(E) R_E/E)^{1/2}\right).\]
Summing, we obtain
\begin{align*} \sum_{E \leq Z} \frac{X^{3/4} \log X}{E^{3/4}} \left(\frac{h_2(E) R_E}{E}\right)^{1/2} & \leq X^{3/4} \log X \left(\sum_{E \leq Z} E^{-1} \right)^{1/2} \left(\sum_{E \leq Z} \frac{h_2(E) R_E}{E^{3/2}} \right)^{1/2} \\
& \ll X^{3/4} \log X \cdot \log Z,
\end{align*}
by Cauchy-Schwarz and partial summation. Again, by taking $Z = X^{1/3 + 1/100}$ say, we easily obtain an acceptable error term. \\

It thus remains to sum over the main terms in Theorem \ref{pos def count}. 

\begin{remark} Theorem \ref{MT} is stated with an error term which is $O(X)$, and this is likely not removable. Indeed, akin to the problem of estimating the sum over the divisor function there is likely a secondary main term of size exactly $X$. Examining the estimation of the summation of errors from Lemmas \ref{BG lem} and \ref{Sh count} above we see that both cases in fact do give power-saving error terms with a suitable choice of $Y$, it is in principle possible to obtain the secondary main term. We wish to return to this problem in future work. 
\end{remark}

\subsection{Summing the main terms} We now consider the two sums
\begin{equation} \label{class sum 1} \sum_{\substack{D \leq  Y \\ D \equiv 3 \pmod{4} \\ w \in W_2^\ast(-D) }} \frac{\pi X}{3 n_w D^{3/2}}
\end{equation}
and
\begin{equation} \label{class sum 2} \sum_{\substack{D \leq  Y \\ D \equiv 0 \pmod{4} \\ w \in W_2^\ast(-D) }} \frac{4 \pi X}{3 n_w D^{3/2}},
\end{equation}
where $n_w$ is given as in Proposition \ref{bad class}. Put
\[h_2^\sharp(-D) = \sum_{\substack{w \in W_2^\ast(-D) \\ n_w = 1 }} 1.\]
We shall then prove the following:

\begin{lemma} \label{negligible amb} The equalities
\begin{equation} \label{class sum one}   \sum_{\substack{D \leq Y \\ D \equiv 3 \pmod{4}}} \frac{h_2^\sharp(-D) \pi X}{3 D^{3/2}} =  \sum_{\substack{D \leq Y \\ D \equiv 3 \pmod{4}}} \frac{h_2(-D) \pi X}{3 D^{3/2}} + O(Y \log Y)
\end{equation} 
and 
\begin{equation} \label{class sum two}  \sum_{\substack{D \leq Y \\ D \equiv 0 \pmod{4}}} \frac{h^\sharp_2(-D) 4 \pi X}{3 D^{3/2}} =\sum_{\substack{D \leq Y \\ D \equiv 0 \pmod{4}}} \frac{h_2(-D) 4 \pi X}{3 D^{3/2}} + O(Y \log Y).
\end{equation}
\end{lemma}

\begin{proof} We have that $n_w > 2$ if and only if $w = [x^2 + xy + y^2]_\bZ$, and it is not possible for a positive definite binary quadratic form $f$ to be opaque. Therefore it suffices to count the number of ambiguous classes with $D \leq Y$ and to show that the number of such classes is small. This requires the estimation of the sum
\[\sum_{D \leq Y} 2^{\omega(D)}.\]
We use the fact that $2^{\omega(n)} = \sum_{d | n} \mu^2(d)$ to obtain
\begin{align*} \sum_{D \leq Y} 2^{\omega(D)} & = Y \sum_{d \leq Y} \mu^2(d) \left(\frac{1}{d} + O(1)\right) \\
& = \frac{8}{27\pi^2} Y \log Y + O(Y).   \end{align*}
Hence the number of ambiguous classes with $D \leq Y$ is $O(Y \log Y)$, as desired. \end{proof}

Lemma \ref{negligible amb} shows that it suffices to estimate the sum 
\begin{equation} \label{penn} \sum_{\substack{D \leq Y \\ D \equiv 3 \pmod{4}}} \frac{h_2(-D) \pi X}{3 D^{3/2}} + \sum_{\substack{D \leq Y \\ D \equiv 0 \pmod{4}}} \frac{h_2(-D) 4 \pi X}{3 D^{3/2}}.
\end{equation}
We can evaluate (\ref{penn}) via Proposition \ref{GauSieg}. Indeed, we note that
\begin{align*} \sum_{\substack{D \leq Y \\ D \equiv 0 \pmod{4}}} \frac{h_2(-D)}{D^{3/2}} & = Y^{-3/2} \sum_{\substack{D \leq Y \\ D \equiv 0 \pmod{4}}} h_2(-D) + \frac{3}{2} \int_1^Y t^{-5/2} \left(\sum_{\substack{D \leq t \\ D \equiv 0 \pmod{4}}} h_2(-D) \right) dt \\
& = \frac{3}{2} \int_1^Y \frac{\pi}{42 \zeta(3)} t^{-1} dt + O(1) \\
& = \frac{\pi \log Y}{28 \zeta(3)} + O(1).
\end{align*}
It thus follows that
\begin{equation} \label{final 1} \sum_{\substack{D \leq Y \\ D \equiv 0 \pmod{4}}} \frac{h_2(-D) 4 \pi X}{3 D^{3/2}} = \frac{\pi^2}{21 \zeta(3)} X \log Y + O(X).
\end{equation}
Similarly, we evaluate
\begin{align*} \sum_{\substack{D \leq Y \\ D \equiv 3 \pmod{4}}} \frac{h_2(-D)}{D^{3/2}} & = Y^{-3/2} \sum_{\substack{D \leq Y \\ D \equiv 3 \pmod{4}}} h_2(-D) + \frac{3}{2} \int_1^Y t^{-5/2} \left(\sum_{\substack{D \leq t \\ D \equiv 3 \pmod{4}}} h_2(-D) \right) dt \\
& = \frac{3}{2} \int_1^Y \frac{2 \pi}{63 \zeta(3)} t^{-1} dt + O(1) \\
& = \frac{\pi}{21 \zeta(3)} \log Y + O(1),
\end{align*}
whence
\begin{equation} \label{final 2} \sum_{\substack{D \leq Y \\ D \equiv 3 \pmod{4}}} \frac{h_2(-D) \pi X}{3 D^{3/2}} = \frac{\pi^2}{21 \zeta(3)} X^{1/3} \log Y + O (X ). \end{equation}
Thus, (\ref{penn}) evaluates to
\begin{equation} \label{final 1} \frac{2\pi^2}{21 \zeta(3)} X \log Y + O(X).
\end{equation} 
Setting $Y = X^{2/3} \exp \left(-4 (\log X) (\log \log X)^{-1}\right)$, we obtain the term
\begin{equation} \label{final 2} \frac{4\pi^2}{63 \zeta(3)} X \log X + O(X).
\end{equation}
We must perform the same analysis for summing over classes of real quadratic forms. Due to the similarity in calculation, we note that the analogue of (\ref{penn}) is
\begin{equation} \label{penn2} \sum_{\substack{E \leq Z \\ E \equiv 0 \pmod{4}}} \frac{4X h_2(E) R_E}{3E^{3/2}} + \sum_{\substack{E \leq Z \\ E \equiv 1 \pmod{4}}} \frac{X h_2(E) R_E}{3E^{3/2}}, 
\end{equation}
and by the second half of Proposition \ref{GauSieg}, we obtain the same conclusion as in the positive definite case. It thus follows that
\begin{equation} N(X) = N(X; Y) + \N(X; XY^{-1}) = \left(\frac{4\pi^2}{63 \zeta(3)} + \frac{2 \pi^2}{63 \zeta(3)}\right) X \log X + O(X) = \frac{2\pi^2}{21 \zeta(3)} X \log X + O(X).
\end{equation}
Finally, by replacing $X$ with $3X^{1/3}/\sqrt[3]{4}$, we complete the proof of Theorem \ref{MT}. 

\section{Proof of Theorem \ref{MT2}}
\label{MT2 proof}

Compared with the proof of Theorem \ref{MT}, the proof of Theorem \ref{MT2} is much simpler, since the reduced classes of reducible binary quadratic forms are particularly simple. Here we find that a typical reducible and reduced binary quadratic form takes the shape
\begin{equation} \label{reduced reducible} f(x,y) = \alpha x^2 + \beta xy, \gcd(\alpha, \beta) = 1.
\end{equation}
This then implies that the lattice $\L_{f,\alpha}$ takes a particularly simple shape, namely
\[\L_{f,\alpha} = \{(x,y) \in \bZ^2 : y \equiv 0 \pmod{4 \alpha^3}\}.\]
We thus replace $B$ with $4 \alpha ^3 B$, so that the generic element $F \in \V_f(\bZ)$ takes the form
\begin{equation} F(x,y) = Ax^4 + 4\alpha^3 Bx^3 y + 6 \alpha^2 \beta B x^2 y^2 + 4 \alpha \beta^2 B xy^3 + \beta^3 B y^4.
\end{equation}
It then follows that $\I(F)$ is given by
\[\I(F) = 4B (4 \alpha^7 B -  \beta A).\]
We then have:
\begin{lemma} \label{red NF} We have
\[N_f(X) = \frac{X^{1/3}}{3\beta^3} \log\left(\frac{X^{1/3}}{3 \beta^2}\right) + \left(2 \gamma - 1 \right) \frac{X^{1/3}}{3\beta^3} + O\left(\frac{X^{1/6}}{\beta}\right).\]
\end{lemma}

\begin{proof} By symmetry, either $B$ or $(4\alpha^7 B - \beta A)$ is less than $X^{1/3}/(12 \beta^2)$ in absolute value. We shall assume that $B, 4 \alpha^7 B - \beta \geq 1$. For convenience, we shall put $Y = X^{1/3}/(12 \beta^2)$. We then look at the three sums
\[S_1(X) = \sum_{m \leq Y^{1/2}} \sum_{\substack{n \leq Y/m \\ n \equiv 4 \alpha^7 m \pmod{\beta}}} 1,\]
\[S_2(X) =\sum_{m \leq Y^{1/2}} \sum_{\substack{n \leq Y/m \\ m \equiv 4 \alpha^7 n \pmod{\beta}}} 1\]
and
\[S_3(X) = \sum_{m \leq Y^{1/2}} \sum_{\substack{n \leq Y^{1/2} \\  n \equiv 4 \alpha^7 m \pmod{\beta}}} 1.\]
It is then clear that
\[N_f(X) = S_1(X) + S_2(X) - S_3(X).\]
We evaluate $S_1(X)$. The inner sum is equal to $\frac{Y}{\beta m} + O(1)$. Thus, we have
\[S_1(X) = \frac{1}{2\beta} \left(Y \log Y + 2\gamma Y \right) + O(Y^{1/2}).\]
Here $\gamma$ is the Euler-Mascheroni constant. The evaluation of $S_2(X)$ is the same, and we have that $S_1(X) = S_2(X) + O(Y^{1/2})$. It is easy to see that
\[S_3(X) = \frac{Y}{\beta} + O(Y^{1/2}).\]
It thus follows that
\begin{align*} N_f(X) & = \frac{Y \log Y + (2 \gamma - 1) Y}{\beta} + O(Y^{1/2}) \\
& = \frac{X^{1/3} \log(X^{1/3}/(12 \beta^2) + (2 \gamma - 1) X^{1/3}}{12 \beta^3} + O\left(\frac{X^{1/6}}{\beta}\right).\end{align*} 
Multiplying by 4 to account for the signs of $B$ and $4 \alpha^7 - \beta A$, we obtain the result. \end{proof}

We may now prove Theorem \ref{MT2}. 

\begin{proof}[Proof of Theorem \ref{MT2}] The error term in the estimate for $N_f(X)$ provided by Lemma \ref{red NF} is sufficiently sharp that we may evaluate the sum directly using the class number formula given by Proposition \ref{red class}. We are then left to evaluate the sum
\[\frac{1}{3} \sum_{n \leq X^{1/6}} \phi(n) \left(\frac{X^{1/3}}{ n^3} \log X - \frac{X^{1/3}}{n^3} \log(12 n^2) + (2 \gamma - 1) \frac{X^{1/3}}{n^3} + O\left(\frac{X^{1/6}}{n}\right) \right). \]
We have the well-known identity
\[\sum_{n \geq 1} \frac{\phi(n)}{n^s} = \frac{\zeta(s-1)}{\zeta(s)}, \Re(s) > 2,\]
Hence 
\[\frac{X^{1/3}}{3}  \left((\log X)/3  + 2 \gamma - 1 \right) \sum_{n \leq X^{1/6}} \frac{\phi(n)}{n^3} = \frac{X^{1/3} \left((\log X)/3  + 2 \gamma - 1 \right)}{3} \left(\frac{\zeta(2)}{\zeta(3)} + O(X^{-1/6})\right). \]
By partial summation, we see that
\[\frac{X^{1/3}}{3} \sum_{n \leq X^{1/6}} \frac{\phi(n) \log(12n^2)}{n^3} = O (X^{1/3}). \]
Since $\phi(n) \leq n -1$ for all positive integers $n$, it follows that
\begin{align*} \sum_{n \leq X^{1/6}} \frac{\phi(n)}{n} & = O(X^{1/6}).
\end{align*}
Finally, we need to address the issue of ambiguous and opaque classes. By Proposition 4.12 in \cite{TX}, we see that the number of classes of discriminant $n^2$ which are opaque or ambiguous is at most $2^{\omega(n) + 1}$. We then note that
\[\sum_{n \leq X^{1/6}} \frac{2^{\omega(n)}}{n^3} = O(1).\] 
Therefore, very few classes are ambiguous or opaque, and the proof is complete. \end{proof}

\section{Proof of Theorems \ref{Gal thm} - \ref{C4 MT}}

\subsection{Proof of Theorem \ref{Gal thm}}

Let $F$ be given by (\ref{quartic form}). Consider its \emph{cubic resolvent polynomial} given by
\[ R_F(x) = a_4^3X^3 - a_4^2a_2X^2 + a_4(a_3a_1 - 4a_4a_0)X - (a_3^2a_0 + a_4a_1^2 - 4a_4a_2a_0).\]
It is well-known that for irreducible $F$, $R_F(x)$ has a rational root if and only if $\Gal(F)$ is isomorphic to a subgroup of $D_4$. In \cite{X} we showed that $R_F(x)$ has a rational root if and only if $F$ has a rational \emph{Cremona covariant}. For $F \in \V_f(\bZ)$, $f$ is a rational Cremona covariant of $F$, hence $R_F(x)$ has a rational root and $\Gal(F)$ is $D_4, C_4$, or $V_4$. \\

We first suppose that $-\Delta(f)$ is not a square. Suppose that $F \in \V_f(\bZ)$. If $\Delta(F)$ is itself a square then $\Gal(F)$ cannot be isomorphic to $C_4$, so we assume that $\Delta(F) \ne \square$. $R_F(x)$ has a unique root $r_F\in\bQ$ precisely when $\Delta(F)\neq\square$, in which case we define
\[\theta_1(F) = (a_3^2 - 4a_4(a_2 - r_Fa_4))\Delta(F)\AND \theta_2(F) = a_4(r_F^2a_4 - 4a_0)\Delta(F).\]
It is well-known (see \cite{Con}) that $\Gal(F) \cong C_4$ precisely when $\Delta(F) \ne \square$ and $\theta_1(F), \theta_2(F)$ are rational squares. Writing $a_4, \cdots, a_0$ as in (\ref{J family}) we find that
\[\theta_1(F) = \frac{-\Delta(f)^3 (\alpha B^2 - 4 \beta AB + 16 \gamma A^2)^4}{16 \alpha^{10}} \]
and
\[\theta_2(F) = \frac{-\Delta(f)^3 (\alpha B^2 - 4 \beta AB + 16 \gamma A^2)^4}{64 \alpha^{12}}.\]
Thus, it is apparent that both $\theta_1(F), \theta_2(F)$ are squares modulo $-\Delta(f)$. Since $-\Delta(f)$ is not a square by assumption, neither are $\theta_1(F), \theta_2(F)$. \\

Now suppose that $-\Delta(f)$ is a square. By the same argument as above, we see that whenever $F$ is irreducible and $\Delta(F)$ is not a square, we have that $\Gal(F) \cong C_4$. It thus remains to show that whenever $\Delta(F)$ is a square, that $F$ is reducible. 

\subsection{Proof of Theorem \ref{V0 inv}} Since all elements in $\V_4^{(0)}(\bR)$ lie in a single $\GL_2(\bR)$-orbit, it suffices to consider the statement for a single element in $\V_4^{(0)}(\bR)$. We take
\[F = xy(x^2 - y^2).\]
It is easily verified that $J(F) = 0$ and $F$ is totally real. Moreover, we have $I(F) = 3$. We compute
\[F_6(x,y) = x^6 - 5x^4 y^2 - 5x^2 y^4 + y^6 = (x^2 + y^2)(x^2 - 2xy - y^2)(x^2 + 2xy - y^2).\]
We then find that
\[G_F = (x^2 - 2xy - y^2)(x^2 + 2xy - y^2).\]
Note that $J(G_F) = 0$ and $G_F$ is also totally real. We then find that
\[(G_F)_6(x,y) = 2(16 - 144)x^5 - 2(16 - 144)xy^5 = 256xy(x^2 + y^2)(x^2 - y^2).\]
Therefore,
\[G_{G_F} = 256 xy(x^2 - y^2),\]
which is proportional to $F$ as claimed. 

\subsection{Proof of Theorem \ref{C4 MT}} Observe that if $-\Delta(f)$ is a square, then it is necessarily even. Hence for each such $f$ we have, by Theorem \ref{pos def count}, 
\[N_f(X) = \frac{4 \pi}{3(4n^2)^{3/2}} X^{1/3} + O\left(\frac{X^{1/6}}{n}\right) = \frac{\pi}{6n^3} X^{1/3} + O \left(\frac{X^{1/6}}{n}\right).\]
The error term from Theorem \ref{pos def count} is not sufficient for our purposes. However, by the same argument given in Section \ref{MT proof} we may first consider all classes of $\nu(f)$, then multiply by the size of the 4-torsion subgroup of the class group. This wins us an extra factor of $\nu_2$, the $x^2$-coefficient of $\nu(f)$, in the denominator. \\

We now consider reduced classes of $\nu(f)$, that is, the set
\[R'(X) = \{(a,b,c) \in \bZ^3 : |b| \leq a \leq c, ac > 0, b^2 - ac = -n^2, 1 \leq n < X^{1/6}/2\}.\]
Let $\ep > 0$. We consider the subset $R''(X)$ of $R'(X)$ with $a \leq n^\ep$. For each $n \in [1, X^{1/6}/2)$ put $h_2^\sharp(-4n^2)$ for the classes counted by $h_2(-4n^2)$ which correspond to an element in $R''(X)$. There are $n^\ep$ choices for $a$, $2a = O(n^\ep)$ choices for $b$, and $d(n^2 + b^2) = O_\ep(n^\ep)$ choices for $c$. Finally, for each such $n$ there are at most $O_\ep(n^\ep)$ elements in the 4-torsion subgroup of the Picard group of $\O_{-4n^2}$. Thus, adjusting $\ep$ if necessary, we have $h_2^\star(-4n^2) = O_\ep\left(n^\ep\right)$. Otherwise we have the bound $h_2(-4n^2) \ll n \log \log n$. Combining these bounds we obtain that summing the error term over all classes counted $R'(X)$ gives an error term of 
\[\sum_{n \leq X^{1/6}/2} O_\ep \left(\frac{X^{1/6}}{n^{1 - \ep}} + \frac{X^{1/6} \log \log n}{n^\ep}\right) = O_\ep \left(X^{1/3 - \ep}\right).\]
We now sum over the main term. The sum over all such $\SL_2(\bZ)$-classes of $f$ gives the sum
\[ \sum_{n \leq X^{1/6}/2} \frac{h_2(-4n^2) \pi}{6n^3} X^{1/3} = X^{1/3} \frac{\pi}{6} \frac{7\zeta(2)}{8\beta(3)} = X^{1/3} \frac{\pi \cdot \pi^2 \cdot 28}{6 \cdot 6 \cdot \pi^3} = \frac{7X^{1/3}}{9}.\]
It thus follows that
\[N_{C_4}(X) = \frac{7X^{1/3}}{9} + O_\ep \left(X^{1/3 - \ep}\right),\]
and the estimate for $M_{C_4}(X)$ follows by replacing $X$ with $4X/27$. \\

Finally, we show that all irreducible elements in $\V_f(\bZ)$ have Galois group $C_4$. We have already ruled out $D_4$, so it suffices to show that all elements $F \in \V_f(\bZ)$ with square discriminant are in fact reducible. Since $-\Delta(f) = \square$, it follows that $f$ is $\GL_2(\bQ)$-equivalent to $x^2 + y^2$. Therefore, $F$ is $\GL_2(\bQ)$-equivalent to a form of the shape
\[\F = Ax^4 + Bx^3 y - 6A x^2 y^2 - Bxy^3 + Ay^4, A,B \in \bQ.\]
By the same argument as in the proof of Theorem \ref{V0 inv}, we find that 
\[G_{\F}(x,y) = (16A^2 + B^2)(Bx^4 - 16Ax^3 y - 6Bx^2 y^2 + 16Axy^3 + By^4). \]
We now suppose that $\Delta(F)$ is a square, which is a necessary condition for $\Gal(F) \cong V_4$. By the proof of Theorem 1.1 in \cite{TX}, we find that $G_{\F}$ is necessarily reducible. Moreover $\Delta(G_{\F})$ is a square, so then it follows that $G_{G_{\F}}$ is necessarily reducible. But $G_{G_{\F}}$ is proportional over $\bQ$ to $F$. Hence $F$ is reducible, as claimed.  

\section*{Appendix: A Coordinate-free Perspective on the Hensel Lifts \\ By Erick Knight}

This appendix is an alternative perspective on the discussion in Section \ref{quadratic lift}; in particular it is a coordinate-free perspective on Proposition \ref{hensel lift prop}.  In this discussion, we will restrict ourselves to considering the class group of a single field $K$. \\

To fix some notation, let $K$ be a quadratic extension of $\mathbb{Q}$, with ring of integers $\mathcal{O}_K$.  Additionally, let $p$ be a prime of $\mathbb{Z}$ that splits in $K$, and write $(p) = \mathfrak{p}_1\mathfrak{p}_2$.  We will denote by $\mathcal{O}_{K, \mathfrak{p}_i}$ to be the $\mathfrak{p}_i$-adic completion of $\mathcal{O}_K$.  Because $p$ splits in $K$, the natural inclusion $\mathbb{Z}_p \hookrightarrow \mathcal{O}_{K, \mathfrak{p}_i}$ is an isomorphism.  Additionally, one has that $\mathcal{O}_K \otimes_\mathbb{Z} \mathbb{Z}_p \cong \mathcal{O}_{K, \mathfrak{p}_1} \oplus \mathcal{O}_{K, \mathfrak{p}_2}$.  Using these isomorphisms, one has that the norm form $N_{K/\mathbb{Q}}: \mathcal{O}_K \otimes_\mathbb{Z} \mathbb{Z}_p \rightarrow \mathbb{Z}_p$ is just the composition $\mathcal{O}_K \otimes_\mathbb{Z} \mathbb{Z}_p \rightarrow \mathcal{O}_{K, \mathfrak{p}_1} \oplus \mathcal{O}_{K, \mathfrak{p}_2} \rightarrow \mathbb{Z}_p \oplus \mathbb{Z}_p \rightarrow \mathbb{Z}_p$ where the first two maps are just the isomorphisms mentioned earlier, and the last map is just $(x, y) \rightarrow xy$. \\

Now, let $f$ be a quadratic form such that $[f]$ is equal to $[\mathfrak{p}_1]$.  This means that there is an identification of $\mathbb{Z}^2$ with $\mathfrak{p}_1$ such that $f$ is equal to the function $\frac{N_{K/\mathbb{Q}}(\cdot)}{p}$ in this basis.  Tensoring up to $\mathbb{Z}_p$, we get that $\mathfrak{p}_1 \otimes_{\mathbb{Z}}\mathbb{Z}_p \cong p\mathcal{O}_{K, \mathfrak{p}_1} \oplus \mathcal{O}_{K, \mathfrak{p}_2}$ as an ideal in $\mathcal{O}_K \otimes_\mathbb{Z} \mathbb{Z}_p$, and $f$ sends an element $(\alpha, \beta)$ of $p\mathcal{O}_{K, \mathfrak{p}_1} \oplus \mathcal{O}_{K, \mathfrak{p}_2}$ to $\alpha\beta/p$. \\

This means that the lattices constructed in Section \ref{quadratic lift} are given by taking the intersection $\mathcal{O}_K \cap p^{s+1}\mathcal{O}_{K, \mathfrak{p}_1} \oplus \mathcal{O}_{K, \mathfrak{p}_2}$ inside of $\mathcal{O}_K \otimes_\mathbb{Z} \mathbb{Z}_p$.  But this is just $\mathfrak{p}_1^{s+1}$, as can be seen from the Chinese remainder theorem.  Moreover, the form $g_{2,k}$ is just given by restriciting $f$ to this lattice and then dividing by $p^s$, which means that $[g_{2,k}]$ is equal to the class of $[\mathfrak{p}_1^{s+1}]$, which is what was wanted.




\end{document}